\documentclass[12pt, a4paper,reqno]{amsart}
\usepackage{amscd,amssymb,amsmath, array,latexsym,amsbsy}
\usepackage[active]{srcltx}
\def\Prim{\operatorname{Prim}}

\def\Orc{\operatorname{Orc}}
\def\coz{\operatorname{coz}}

\def\C{\mathbb{C}}

\newtheorem{thm}{Theorem}[section]
\newtheorem{cor}[thm]{Corollary}
\newtheorem{prop}[thm]{Proposition}
\newtheorem{lemma}[thm]{Lemma}
\theoremstyle{definition}

\newtheorem{example}[thm]{Example}
\numberwithin{equation}{section}
\title{Spectral synthesis in the multiplier algebra of a $C_0(X)$-algebra}

\author
[Archbold]{Robert J. Archbold}
\address{Institute of Mathematics
\\University of Aberdeen
\\King's College
\\Aberdeen AB24 3UE
\\Scotland
\\United Kingdom
} \email{r.archbold@abdn.ac.uk}

\author[Somerset]{Douglas W.B. Somerset}
\address{Institute of Mathematics
\\University of Aberdeen
\\King's College
\\Aberdeen AB24 3UE
\\Scotland
\\United Kingdom}
 \email{somerset@quidinish.fsnet.co.uk}

\begin{document}

\begin{abstract}
Let $A$ be a $C_0(X)$-algebra with continuous map $\phi$ from ${\rm
Prim}(A)$, the primitive ideal space of $A$, to a locally compact
Hausdorff space $X$. Then the multiplier algebra $M(A)$ is a
$C(\beta X)$-algebra with continuous map $\overline\phi: {\rm
Prim}(M(A))\to\beta X$ extending $\phi$. For $x\in {\rm Im}(\phi)$,
let $J_x=\bigcap\{ P\in {\rm Prim}(A): \phi(P)=x\}$ and
$H_x=\bigcap\{ Q\in {\rm Prim}(M(A)):\overline\phi(Q)=x\}$. Then
$J_x\subseteq H_x\subseteq\tilde J_x$, the strict closure of $J_x$
in $M(A)$. Thus $H_x$ is strictly closed if and only if $H_x=\tilde
J_x$, and the `spectral synthesis' question asks when this happens.
In this paper it is shown that, for $\sigma$-unital $A$, $H_x$ is
strictly closed for all $x\in {\rm Im}(\phi)$ if and only if $J_x$
is locally modular for all $x\in {\rm Im}(\phi)$ and $\phi$ is a
closed map relative to its image. Various related results are
obtained.
\end{abstract}

\maketitle

\thanks{}


%\magnification=\magstep1 \baselineskip=16pt
%\input amssym.tex
%\input spectral.ref
%\def\Prim{{\rm Prim}}
%\def\Glimm {{\rm Glimm}}
%\def\Sub {{\rm Sub}}
%\def\Orc{{\rm Orc}}
%\def\coz{{\rm coz}}
%\centerline {\bf SPECTRAL SYNTHESIS IN THE MULTIPLIER ALGEBRA}
%\centerline {\bf OF A $C_0(X)$-ALGEBRA}

%\centerline {\bf R.J. Archbold and D.W.B. Somerset}

%\noindent {\bf 2000 Maths Subject Classification} (46L05, 46L57)

\bigskip

\section{Introduction}

 Let $A$ be a C$^*$-algebra with multiplier algebra $M(A)$ \cite{APT, Bus}.
The ideal structure of $M(A)$ is typically much more complicated
than that of $A$ and has been widely studied by a number of authors
{\cite{AE, Elli, Kuc, Li, Lin, Per, Ror}. In \cite{Xalg, multid} we
investigated certain aspects of the ideal structure of M(A) which
arise from a $C_0(X)$-algebra structure on the algebra $A$. The
notion of a $C_0(X)$-algebra was introduced by Kasparov \cite{Kasp}
following extensive earlier work on continuous and upper
semi-continuous fields (see, for example, \cite{DH, Fel, Glimm, Le,
Tom}). As noted in \cite[Section 1]{multid}, every $C^*$-algebra is
a $C_0(X)$-algebra, typically in many ways. $C_0(X)$-algebras have
been studied in \cite{AM, Blanch, Dup, DupGil, EWi, HRW, Nilsen}. In
\cite[Section 1]{Xalg}, we saw that if $A$ is a $C_0(X)$-algebra
then corresponding to each $x\in X$ there are two natural ideals
 $H_x$ and $\tilde{J_x}$ in $M(A)$ which one might hope would be
equal but in fact need not be so. The aim of this paper is to
characterize the equality $H_x=\tilde{J_x}$ in terms of $A$ and $X$
(without reference to $M(A)$).

 Recall
that $A$ is a \emph{$C_0(X)$-algebra} if there is a continuous map
$\phi$, called the \emph{base map}, from ${\rm Prim}(A)$, the
primitive ideal space of $A$ with the hull-kernel topology, to the
locally compact Hausdorff space $X$ \cite[Proposition C.5]{Will}. We
will use $X_{\phi}$ to denote the image of $\phi$ in $X$. If $A$ is
a $C_0(X)$-algebra then $M(A)$ is a $C(\beta X)$-algebra with
continuous map $\overline\phi: {\rm Prim}(M(A))\to \beta X$ (the
Stone-\v{C}ech compactification of $X$) extending $\phi$ (see
Section 2).

For $x\in X_{\phi}$, let $J_x$ be the closed ideal of $A$ defined by
$J_x=\bigcap\{ P\in {\rm Prim}(A): \phi(P)=x\}$ and let $H_x$ be the
closed ideal of $M(A)$ defined by $H_x=\bigcap\{ Q\in {\rm
Prim}(M(A)):\overline\phi(Q)=x\}$. Let $\tilde J_x$ be the strict
closure of $J_x$ in $M(A)$. Then $J_x\subseteq H_x\subseteq\tilde
J_x$ and hence $H_x$ is strictly closed if and only if $H_x=\tilde
J_x$ (see Proposition~\ref{Prop 1.3}). The `spectral synthesis'
question asks for conditions on $A$ and $X$ characterizing when
$H_x$ is strictly closed. It was shown in \cite{multid}, for
instance, that if $A$ is stable and $\sigma$-unital then for $x\in
X_{\phi}$, $H_x$ is strictly closed if and only if $x$ is a P-point
in $X_{\phi}$.

In this paper we return to the question in a more general context.
The main result (Corollary~\ref{Cor 2.7}) is that if $A$ is a
$\sigma$-unital $C_0(X)$-algebra with base map $\phi$ then $H_x$ is
strictly closed for all $x\in X_{\phi}$ if and only if $J_x$ is
locally modular for all $x\in X_{\phi}$ and $\phi$ is a closed map
relative to its image. Here $J_x$ is said to be {\sl locally
modular} if whenever $Q$ lies in the boundary in $\Prim(A)$ of
$H(x)=\{ P\in\Prim(A): P\supseteq J_x\}$
 then there exists a neighbourhood $V$ of $Q$ in $\Prim(A)\setminus U(x)$
(where $U(x)$ is the interior of $H(x)$) such that $A/\ker V$ is a
unital $C^*$-algebra.

If $A$ is separable then the same characterization is valid for
spectral synthesis at a point $x\in X_{\phi}$ (Corollary~\ref{Cor
2.10}), namely $H_x$ is strictly closed if and only if $J_x$ is
locally modular and $\phi$ is a locally closed map at $x$. For
general $\sigma$-unital $C_0(X)$-algebras this condition is close to
characterizing spectral synthesis at a point (see Theorem~\ref{Thm
2.6}) but does not quite succeed (Example~\ref{Ex 3.5}), and we have
to leave the problem open.

The structure of the paper is as follows. Section 2 gives basic
information on $C_0(X)$-algebras and spectral synthesis. Section 3
contains some of the main results of the paper, described above.
Section 4 looks more closely at pointwise spectral synthesis for a
$\sigma$-unital $C_0(X)$-algebra $A$ and identifies the points in
$X_{\phi}$ which are difficult to deal with. The main result is the
characterization of pointwise spectral synthesis in the case when
$A$ is a continuous $C_0(X)$-algebra (Theorem~\ref{Thm 3.6}).

In Sections 5 and 6, we restrict to the important special case when
$\phi$ is the complete regularization map for ${\rm Prim}(A)$ and
the connecting order $\Orc(A)$ is finite. In Section 5, we show
that, in this case, if $A$ is $\sigma$-unital then the local
modularity of $J_x$ implies that the complete regularization map
$\phi$ is locally closed at $x$. Hence if $A$ is also separable then
$H_x$ is strictly closed if and only if $J_x$ is locally modular
(Corollary~\ref{Cor 4.4}). In Section 6, we show that if $J_x$ is
locally modular then either $J_x$ does not contain the centre of $A$
or the hull $H(x)$ of $J_x$ in $\Prim(A)$ must have non-empty
interior---an unusual property unless $H(x)$ is a clopen set
(Corollary~\ref{Cor 5.3}).

\bigskip

\noindent Note added in revision.

After this paper was submitted for publication, we learned that
David McConnell (private communication) had independently obtained
Proposition~\ref{Prop 1.7}(iii), Proposition~\ref{Prop 1.7}(i) (in
the case where $\phi$ is the complete regularization map), and
versions of Corollary~\ref{Cor 3.3}, Corollary~\ref{Cor 3.4} and
Theorem~\ref{Thm 3.6} with somewhat stronger hypotheses.

 We are grateful to the referee for a number of helpful comments and
 for pointing out an error in the original proof of Theorem~\ref{Thm
 2.6}.

%\bigskip

%\bigskip

\section{General $C_0(X)$-algebras}

%\bigskip

 In this section we gather some information about $C_0(X)$-algebras
and establish the basic facts about spectral synthesis
(Proposition~\ref{Prop 1.7}). For $C_0(X)$-algebras we follow the
terminology of \cite{Xalg}.

Let $A$ be a $C_0(X)$-algebra with base map $\phi: {\rm Prim}(A)\to
X$, and recall that $X_{\phi}={\rm Im}(\phi)$. Then $X_{\phi}$ is
completely regular; and if $A$ is $\sigma$-unital, $X_{\phi}$ is
$\sigma$-compact and hence normal \cite[Section 2]{multid}. For
$x\in X_{\phi}$, set $J_x=\bigcap \{ P\in {\rm Prim}(A):
\phi(P)=x\}$, and for $x\in X\setminus X_{\phi}$, set $J_x=A$. For
$a\in A$, the function $x\to \Vert a+J_x\Vert$ $(x\in X)$ is upper
semi-continuous \cite[Proposition C.10]{Will}. The $C_0(X)$-algebra
$A$ is said to be {\sl continuous} if, for all $a\in A$, the norm
function $x\to \Vert a+J_x\Vert$ $(x\in X$) is continuous. By Lee's
theorem \cite[Proposition C.10 and Theorem C.26]{Will}, this happens
if and only if the base map $\phi$ is open.

Let $J$ be a proper, closed, two-sided ideal of a $C^*$-algebra $A$.
The quotient map $q_J: A \to A/J$ has a canonical extension
$\tilde{q_J}:M(A)\to M(A/J)$. We define a proper, closed, two-sided
ideal $\tilde J$ of $M(A)$ by
$$\tilde J = \ker \tilde{q_J} = \{b\in M(A): ba,ab\in
J\hbox{ for all } a\in A\}.$$ The following proposition was proved
in \cite[Proposition 1.1]{Xalg}.

\medskip

\begin{prop} \label{Prop 1.1}  Let $J$ be a proper, closed,
two-sided ideal of a $C^*$-algebra $A$. Then

{\rm (i)} $\tilde J$ is the strict closure of $J$ in $M(A)$;

{\rm (ii)} $\tilde J \cap A = J$;

{\rm (iii)} if $P\in {\rm Prim}(A)$ then $\tilde P$ is primitive
(and hence is the unique ideal in ${\rm Prim}(M(A))$ whose
intersection with $A$ is $P$);

{\rm (iv)} $\tilde J = \bigcap\{\tilde P: P\in {\rm Prim}(A)\hbox{
and } P\supseteq J\}$ and for all $b\in M(A)$ $$ \Vert
b+\tilde{J}\Vert = \sup\{\Vert b+\tilde P\Vert: P\in {\rm
Prim}(A)\hbox{ and } P \supseteq J\};$$

{\rm (v)} $(A+\tilde J)/\tilde J$ is an essential ideal in
$M(A)/\tilde J$.
\end{prop}

%\smallskip

\noindent Furthermore, the map $P\mapsto\tilde P$ $(P\in {\rm
Prim}(A))$ maps ${\rm Prim}(A)$ homeomorphically onto a dense, open
subset of ${\rm Prim}(M(A))$ \cite[4.1.10]{GKP}. For $S\subseteq
{\rm Prim}(A)$, we write $\tilde S=\{ \tilde P: P\in S\}$. In view
of Proposition~\ref{Prop 1.1}(ii), $(A+\tilde{J})/\tilde{J}$ is
canonically isomorphic to $A/J$. If $A/J$ is unital then
$(A+\tilde{J})/\tilde{J}$ is a unital essential ideal of
$M(A)/\tilde{J}$ and therefore equal to $M(A)/\tilde{J}$.

Now suppose that $A$ is a $C_0(X)$-algebra, $x\in X$ and $a\in A$.
If $A/J_x$ is unital, the spectrum of $a+J_x$ (in $A/J_x$) coincides
with the spectrum of $a+\tilde{J}_x$ in $M(A)/\tilde{J}_x$ by the
previous remark. If $A/J_x$ is non-unital, the spectrum of $a+J_x$
(in the unitization of $A/J_x$) is equal to the spectrum of
$a+\tilde{J}_x$ in $(A+\tilde{J}_x)/\tilde{J}_x+\C(1+\tilde{J_x})$
and hence in $M(A)/\tilde{J}_x$ \cite[1.3.10(ii)]{Dix}.

The following proposition was proved in \cite[Proposition
1.2]{Xalg}.

%\bigskip

\begin{prop} \label{Prop 1.2}  Let $A$ be a $C_0(X)$-algebra with
base map $\phi$. Then $\phi$ has a unique extension to a continuous
map $\overline{\phi}: {\rm Prim}(M(A))\to \beta X$ such that
$\overline{\phi}(\tilde P) = \phi(P)$ for all $P\in {\rm Prim}(A)$.
Hence $M(A)$ is a $C(\beta X)$-algebra with base map $\overline\phi$
and ${\rm Im}(\overline\phi)={\rm cl}_{\beta X}(X_{\phi})$.
\end{prop}

%\bigskip

\noindent Now let $A$ be a $C_0(X)$-algebra with base map $\phi$ and
let $\overline{\phi}: {\rm Prim}(M(A))\to \beta X$ be as in
Proposition~\ref{Prop 1.2}. For $x\in \beta X$, we define $H_x =
\bigcap \{ Q\in {\rm Prim}(M(A)): \overline\phi (Q)=x\}$, a closed
two-sided ideal of $M(A)$. Thus $H_x$ is defined in relation to
$(M(A),\beta X,\overline{\phi})$ in the same way that $J_x$ (for
$x\in X)$ is defined in relation to $(A,X,\phi)$. It follows that
for each $b\in M(A)$, the function $x\to\Vert b+H_x\Vert$ $(x\in
\beta X)$ is upper semi-continuous.

%\bigskip

 The next proposition was proved in \cite[Proposition 1.3]{multid}.

%\bigskip

\begin{prop} \label{Prop 1.3}  Let $A$ be a $C_0(X)$-algebra with
base map $\phi$, and set $X_{\phi}={\rm Im}(\phi)$.

{\rm (i)} For all $x\in X$, $J_x\subseteq H_x\subseteq \tilde{J_x}$
and $J_x = H_x\cap A$.

{\rm (ii)} For all $x\in X$, $H_x$ is strictly closed if and only if
$H_x=\tilde J_x$.

{\rm (iii)} For all $b\in M(A)$, $\Vert b\Vert= \sup\{\Vert
b+\tilde{J_x}\Vert:x\in X_{\phi}\}=\sup\{\Vert b+H_x\Vert:x\in
X_{\phi}\}$.
\end{prop}

%\bigskip

%\noindent For a $C_0(X)$-algebra $A$ with base map $\phi$,
%it is evident that the ideals $J_x$ and $\tilde J_x$ $(x\in X_{\phi})$ depend only on $\phi$
%and are independent of the ambient space $X$.
%The next lemma, which was proved in [\the\multid; Lemma 1.4]
%shows that the ideal $H_x$ $(x\in X_{\phi})$ is also
%independent of $X$. In this lemma, for $f\in C^b({\rm Prim}(A))$,
%$\overline f$ denotes the Dauns-Hofmann extension of
%$f$ to a continuous function on ${\rm Prim}(M(A))$ such that $\overline{f}(\tilde{P})=f(P)$.

%\noindent {\bf Lemma 1.4} {\sl Let $A$ be a $C_0(X)$-algebra with base map $\phi$ and set
%$X_{\phi}={\rm Im}(\phi)$. Let $x\in X_{\phi}$ and let $Q\in {\rm Prim}(M(A))$.
%Then $Q\supseteq H_x$ if and only if $\overline {f\circ \phi}(Q)= f(x)$ for
%all $f\in C^b(X_{\phi})$. Hence $H_x$ depends only on $\phi$ and $X_{\phi}$ and
%is independent of the ambient space $X$.}

 We now turn to the subject of spectral synthesis and our first
proposition seeks to justify the use of the name. Recall that in the
the theory of commutative Banach algebras, spectral synthesis holds
at a point $x$ in the maximal ideal space provided that each element
of the algebra whose Gelfand transform vanishes at $x$ can be
approximated in (the original) norm by elements whose transforms
vanish in a neighbourhood of $x$. At this stage we need some more
notation.

For a C$^*$-algebra $A$, let $Z(A)$ denote the centre of $A$.
Now let $A$ be a $C_0(X)$-algebra with base map $\phi$. For $b\in M(A)$,
let $Z(b)=\{ x\in X_{\phi}: b\in \tilde J_x\}$ and let
${\rm Int}\ Z(b)$ be the interior of $Z(b)$ relative to $X_{\phi}$.
%Let $\mu:
%C_0(X)\to Z(M(A))$ be given by $\mu (f)= \theta_A (f\circ\phi)$
%$(f\in C_0(X))$
 Recall that the Dauns-Hofmann isomorphism $\theta_A: C^b({\rm Prim}(A))\to Z(M(A))$
 has the property that $\theta_A(f)a+P=f(P)(a+P)$ for $f\in
C^b(\Prim(A))$, $a\in A$, and $P\in \Prim(A)$ (equivalently,
$\theta_A(f)+\tilde{P}=f(P)(1+\tilde{P})$).
%We also define $\overline\mu:
%C(\beta X)\to Z(M(A))$ by $\overline\mu (f)= \theta_A (f\circ\phi)$
%$(f\in C(\beta X))$ (see [\the\Xalg] for
%more on the maps $\mu$ and $\overline\mu$).

\bigskip

\begin{prop} \label{Prop 1.5}  Let $A$ be a $C_0(X)$-algebra with
base map $\phi$ and let $x\in X_{\phi}$. Set $H^{\rm alg}_x=\{ b\in
M(A): x\in {\rm Int}\ Z(b)\}$. Then $H_x^{\rm alg}\subseteq H_x$ and
$H_x$ is the norm-closure of $H_x^{\rm alg}$.
\end{prop}

\begin{proof} Let $b\in H^{\rm alg}_x$. Then $x$ lies in the
interior $U$ of $Z(b)$ in $X_{\phi}$. There exists $f\in
C^b(X_{\phi})$ such that $f(x)=0$ and $f(X_{\phi}\setminus U)=
\{1\}$. Let $z=\theta_A(f\circ\phi)\in Z(M(A))$. Suppose that
$Q\in\Prim(M(A))$ and $Q\supseteq H_x$. Let $(P_{\alpha})$ be a net
in $\Prim(A)$ such that $\tilde{P}_{\alpha}\to Q$. Since $z$ is
central,
$$\Vert z+Q\Vert = \lim\Vert z+\tilde{P}_{\alpha}\Vert
=\lim|f(\phi(P_{\alpha}))| = |f(\overline{\phi}(Q))| =|f(x)|=0.$$
Thus $z\in Q$ and hence $z\in H_x$.
% Let $z\in Z(M(A))$ such that
%$z+Q=\overline{f\circ \phi}(Q)$ for all $Q\in {\rm Prim}(M(A))$.
%Then $z+Q=f(x)=0$ if $Q\supseteq H_x$ by Lemma 1.4, and hence $z\in
%H_x$.
For $P\in {\rm Prim}(A)$, $$zb+\tilde P=f(\phi(P))(b+\tilde
P)=b+\tilde P$$ because $b\in\tilde{J}_{\phi(P)}\subseteq\tilde{P}$
whenever $f(\phi(P))\neq1$.  Hence $zb=b$ and therefore $b\in H_x$.

For the second part of the proof, let $b\in H_x$ and $\epsilon>0$.
By upper semi-continuity, there is an open neighbourhood $U$ of $x$
in $X$ such that $\Vert b+H_y\Vert <\epsilon$ for all $y\in U$. Let
$V=U\cap X_{\phi}$. There exists a continuous function
$f:X_{\phi}\to[0,1]$ such that $f(x)=0$ and $f(X_{\phi}\setminus V)
=\{1\}$. Define $g:[0,1]\to[0,1]$ by $g(t)=0$ ($0\leq t\leq 1/2$)
and $g(t)=2t-1$ ($1/2<t\leq 1$). Let $h=g\circ f$. Then
$h(X_{\phi}\setminus V) =\{1\}$ and there exists an open
neighbourhood $W$ of $x$ in $X_{\phi}$ such that $h|_W=0$. Let
$z=\theta_A(h\circ\phi)\in Z(M(A))$.

Suppose that $y\in W$, $P\in\Prim(A)$ and $\phi(P)=y$. Then
$zb+\tilde{P} = h(y)(b+\tilde{P}) =0$. It follows that
$zb\in\tilde{J}_y$ and hence $zb\in H^{\rm alg}_x$.

Finally, for $P\in\Prim(A)$,
$$\Vert (b-bz)+\tilde{P}\Vert = (1-h(\phi(P)))\Vert b+\tilde{P}\Vert
\leq (1-h(\phi(P)))\Vert b+H_{\phi(P)}\Vert<\epsilon$$ since
$h(\phi(P))=1$ if $\phi(P)\notin U$. Hence $\Vert b-zb\Vert <
\epsilon$.
\end{proof}
%Recall that an equivalent definition of $H_x$ (see [\the\Xalg])
%is that $H_x=\overline\mu\{ f\in C(\beta X): f(x)=0\}M(A)$. Hence
%by the Cohen Factorization theorem we may write $b=zc$ where
%$c\in M(A)$ and $z=\overline\mu(f)$ for some $f\in C(\beta X)$
%with $f(x)=0$. Let $\epsilon>0$ and let $g\in C(\beta X)$ with $\Vert f-g\Vert<\epsilon$
%such that $g$ vanishes in a neighbourhood $U$ of $x$ in $\beta X$. Set
%$d=\overline\mu(g)c$. Then $\Vert d-b\Vert<\epsilon$ and $Z(d)\supseteq Z(\overline\mu(g))$.
%Let $y\in U\cap X_{\phi}$ and let $P\in {\rm Prim}(A)$ with $\phi(P)=y$.
%Then $\overline\mu(g)+\tilde P=\theta_A (g\circ \phi) +\tilde P=g(\phi(P))=g(y)=0$. Hence
%$\overline\mu(g)\in \tilde J_y$, so $y\in Z(\overline\mu(g))$.
%Thus $Z(d)\supseteq U$, so $d\in H^{\rm alg}_x$. Q.E.D.

It follows from Proposition~\ref{Prop 1.3} and Proposition~\ref{Prop
1.5} that, for $x\in X_{\phi}$, the ideal $H_x$ is strictly closed
if and only if every element $b\in M(A)$ which vanishes at $x$ can
be approximated in norm by elements vanishing in a neighbourhood of
$x$ in $X_{\phi}$. Thus we think of $H_x$ being strictly closed as
corresponding to `spectral synthesis holding at $x$'.

As in \cite{Xalg}, we define $\mu: C_0(X)\to Z(M(A))$ by $\mu (f)=
\theta_A (f\circ\phi)$ $(f\in C_0(X))$
 Now set $Z'(A)=\mu(C_0(X))\cap A$ and note that
$Z'(A)\subseteq Z(A)$. The next lemma was proved in \cite[Lemma
2.1]{Xalg}.

\begin{lemma} \label{Lemma 1.6}  Let $A$ be a $C_0(X)$-algebra with base
map $\phi$ and let $x\in X_{\phi}$. Then $J_x\supseteq Z'(A)$ if and
only if there exists $R\in {\rm Prim}(M(A))$ with $R\supseteq A$
such that $\overline\phi(R)=x$.
\end{lemma}

\noindent We can now give the basic results on spectral synthesis
for general $C_0(X)$-algebras.

%\bigskip

\begin{prop} \label{Prop 1.7}  Let $A$ be a $C_0(X)$-algebra with
base map $\phi$ and let $x\in X_{\phi}$.

{\rm (i)} If $J_x\not\supseteq Z'(A)$ then $H_x$ is strictly closed
in $M(A)$.

{\rm (ii)} If $x$ is an isolated point in $X_{\phi}$ then $H_x$ is
strictly closed in $M(A)$.

{\rm (iii)} If $J_x\supseteq Z'(A)$ and $A/J_x$ is unital then $H_x$
is not strictly closed in $M(A)$.
\end{prop}

\begin{proof} (i) This was proved in \cite[Proposition
2.2]{Xalg}.

(ii) Since $x$ is an isolated point in $X_{\phi}$, $W=\phi^{-1}(x)$
is a clopen subset of ${\rm Prim}(A)$. Set $Y={\rm Prim}(A)\setminus
W$. Suppose that $R\in {\rm Prim}(M(A))$ and $R\not\supseteq\tilde
J_x$. Then $R$ is not in the closure of the set $\{\tilde{P}: P\in
W\}$. Since the set $\{\tilde P: P\in {\rm Prim}(A)\}$ is dense in
${\rm Prim}(M(A))$, $R$ must lie in the closure of the set $\{
\tilde P: P\in Y\}$. Hence the continuity of $\overline\phi$ implies
that $\overline\phi(R)$ lies in the closure of $\{\phi(P): P\in
Y\}$. Thus $\overline\phi(R)\in X \setminus \{x\}$ and so
$R\not\supseteq H_x$.

(iii) Let $p\in A$ such that $p+J_x$ is the identity for $A/J_x$.
Then $ap-a,pa-a\in J_x$ for all $a\in A$ and hence $1-p\in\tilde
J_x$. On the other hand, by Lemma~\ref{Lemma 1.6} there exists $R\in
{\rm Prim}(M(A))$ with $R\supseteq A$ such that
$\overline\phi(R)=x$. Since $R\supseteq A$, $p\in R$, and hence
$1-p\notin R$. Thus $1-p\notin H_x$.
\end{proof}

 If $J_x \not\supseteq Z'(A)$ then $A/J_x$ is unital
\cite[Proposition 2.2]{Xalg}. So the three cases of
Proposition~\ref{Prop 1.7} cover all possibilities for $x$ except
when $x$ is a non-isolated point of $X_{\phi}$ with $A/J_x$
non-unital. This is the case of interest which will occupy us for
the rest of the paper. We will use the following notation. Let
$U_{\phi}=\{x\in X_{\phi}: J_x\not\supseteq Z'(A)\}$, an open subset
of $X_{\phi}$ (see \cite[Section 2]{Xalg}); and let
$W_{\phi}=X_{\phi}\setminus U_{\phi}$. Let $\partial U_{\phi}$
denote the boundary of $U_{\phi}$ in $X_{\phi}$.

%\bigskip

 Finally in this section, we consider spectral synthesis for a
closed subset $E$ of $X_{\phi}$. Define $J_E=\bigcap\{ P\in {\rm
Prim}(A): \phi(P)\in E\}$, $H_E=\bigcap\{ Q\in {\rm Prim}(M(A)):
\overline\phi(P)\in E\}$ and $H^{\rm alg}_E=\{ b\in M(A): E\subseteq
{\rm Int}\ Z(b)\}$. In the following analogue of
Proposition~\ref{Prop 1.5}, we restrict to the case where $A$ is
$\sigma$-unital in order to ensure that $X_{\phi}$ is normal.

\begin{prop}\label{Prop 1.7B}  Let $A$ be a $\sigma$-unital
$C_0(X)$-algebra with base map $\phi$ and let $E$ be a closed subset
of  $X_{\phi}$. Then $H_E^{\rm alg}\subseteq H_E$ and $H_E$ is the
norm-closure of $H_x^{\rm alg}$.
\end{prop}

\begin{proof} Let $b\in H^{\rm alg}_E$. Then $E$ is contained
in the interior $U$ of $Z(b)$ in $X_{\phi}$. Since $X_{\phi}$ is
normal, there exists $f\in C^b(X_{\phi})$ such that $f(E)=\{0\}$ and
$f(X_{\phi}\setminus U)= \{1\}$. Let $z=\theta_A(f\circ\phi)\in
Z(M(A))$. As in the proof of Proposition~\ref{Prop 1.5}, $z\in H_E$
and $b=zb\in H_E$.

For the second part of the proof, let $b\in H_E$ and $\epsilon>0$.
By upper semi-continuity, there is an open neighbourhood $U$ of $E$
in $X$ such that $\Vert b+H_y\Vert <\epsilon$ for all $y\in U$. Let
$V=U\cap X_{\phi}$. Since $X_{\phi}$ is normal, there exists an open
neighbourhood $W$ of $E$ in $X_{\phi}$, with closure contained in
$V$, and a continuous function $f:X_{\phi}\to[0,1]$ such that
$f(W)=\{0\}$ and $f(X_{\phi}\setminus V) =\{1\}$.  Let
$z=\theta_A(f\circ\phi)\in Z(M(A))$. As in the proof of
Proposition~\ref{Prop 1.5}, $zb\in H^{\rm alg}_x$ and
 $\Vert b-zb\Vert <
\epsilon$.
\end{proof}

\bigskip

Let $A$ be a $\sigma$-unital $C_0(X)$-algebra with base map $\phi$
and let $E$ be a closed subset of  $X_{\phi}$. If $P\in\Prim(A)$ and
$P\supseteq J_E$ then $\phi(P)\in E$ because $\phi$ is continuous
and $E$ is closed. Hence $$J_E =\bigcap_{x\in E} J_x  \subseteq
\bigcap_{x\in E} H_x =H_E\subseteq \bigcap_{x\in E} \tilde{J}_x
=\tilde{J}_E,$$ by Proposition~\ref{Prop 1.1}(iv). It follows from
Proposition~\ref{Prop 1.1}(i) and Proposition~\ref{Prop 1.7B} that
  $H_E^{\rm alg}$ is dense in $\bigcap_{x\in E} \tilde{J}_x$
  (which may be thought of as spectral synthesis for E) if and only if
  $H_E$ is strictly
closed in $M(A)$. We therefore define {\sl spectral synthesis for}
$E$ to mean that $H_E$ is strictly closed in $M(A)$. If $E$ and $F$
are closed subsets of $X_{\phi}$ then $H_{E\cup F}=H_E\cap H_F$. In
contrast to the theory of commutative Banach algebras, it follows
that if $E$ and $F$ have spectral synthesis then so does $E\cup F$.

\bigskip

We shall see in Proposition~\ref{Prop 1.9} that the question of
spectral synthesis for a closed subset $E$ of $X_{\phi}$ (i.e.
whether $H_E$ is strictly closed in $M(A)$) can be reduced to the
question of spectral synthesis for a singleton, but at the expense
of changing the base map $\phi$. We will need the following standard
topological lemma, where $X/E$ is the quotient space of a
topological space $X$ obtained by identifying all of the points in a
given subset $E$.

 \bigskip

\begin{lemma} \label{Lemma 1.8} Let $X$ be a normal, Hausdorff space and
let $E$ be a non-empty closed subset. Set $Y=X/E$. Then $Y$ is
normal and Hausdorff and hence completely regular.
\end{lemma}

\begin{proof} Let $q:X\to Y$ be the quotient map. For $y\in Y$,
$q^{-1}(y)$ is closed and so $Y$ is a $T_1$ space. Let $B$ and $C$
be disjoint, non-empty closed subsets of $Y$. Then $G=q^{-1}(B)$ and
$H=q^{-1}(C)$ are disjoint closed sets in the normal space $X$, and
hence there exist disjoint open sets $U$ and $V$ such that
$G\subseteq U$ and $H\subseteq V$. Without loss of generality, we
may assume that $G\not\supseteq E$ and hence $G\cap E=\emptyset$.
Thus, replacing $U$ by $U\setminus E$, we may assume that $U$ does
not meet $E$. Hence $U$ is saturated with respect to the equivalence
relation corresponding to $q$ and so $q(U)$ is an open neighbourhood
of $B$ and is disjoint from $q(V)$. If $H\supseteq E$ then $V$ is
saturated and so $q(V)$ is open, and if $H\not\supseteq E$ then
$V\setminus E$ is a saturated open set containing $H$ and so
$q(V\setminus E)$ is an open neighbourhood of $C$. Hence $Y$ is
normal, and being $T_1$, it is also Hausdorff and completely
regular.
\end{proof}

\noindent Note that if $X$ is Hausdorff but non-normal then $X$ has a closed set
$E$ such that $X/E$ is not completely regular. To see this, let $E$ and $F$ be disjoint
closed subsets of $X$ which cannot be separated by disjoint open sets. Let $Y=X/E$
with quotient map $q$ and set $q(E)=e$. Then $e$ and the closed set $q(F)$ cannot
be separated by disjoint open sets, so $Y$ is not even regular.

%\bigskip

 Now let $A$ be a $C_0(X)$-algebra with base map $\phi$. Let $E$ be
a non-empty closed subset of $X_{\phi}$.  Set $Y=X_{\phi}/E$ and let
$q: X_{\phi}\to Y$ be the quotient map. If $A$ is $\sigma$-unital
then $Y$ is completely regular by Lemma~\ref{Lemma 1.8} and so $A$
is a $C(\beta Y)$-algebra with base map $\psi=q\circ \phi$. Let
$e=q(E)$. The next proposition relates $J_E$ and $J_e$, and $H_E$
and $H_e$.

Note that if $f\in C^b(\Prim(A))$ then the element $\theta_A(f)\in
Z(M(A))$ induces a function $\overline{f}\in C(\Prim(M(A))$ such
that
$$ \theta_A(f)+Q = \overline{f}(Q)(1+Q) \qquad (Q\in\Prim(M(A))).$$
In particular, $\overline{f}(\tilde{P}) = f(P)$ for all
$P\in\Prim(A)$.

%\bigskip

\begin{prop} \label{Prop 1.9}  Let $A$ be a $\sigma$-unital
$C_0(X)$-algebra with base map $\phi$ and let $E$ be a non-empty
closed subset of $X_{\phi}$. Set $Y=X_{\phi}/E$ and let $q:
X_{\phi}\to Y$ be the quotient map. Let $\psi=q\circ \phi$. Then $A$
is a $C(\beta Y)$-algebra with base map $\psi$, and $J_e=J_E$ and
$H_e=H_E$, where $\{e\}=q(E)$.
\end{prop}

%\bigskip

\begin{proof} For $P\in {\rm Prim}(A)$, $\phi(P)\in E$ if and only if
$\psi(P)=e$ and therefore $J_E=J_e$.

For $P\in \Prim(A)$, $\overline{q\circ \phi}(\tilde{P}) =
(q\circ\phi)(P) = \overline{q}(\overline{\phi}(\tilde{P}))$ and
hence $\overline{q\circ\phi} = \overline{q}\circ\overline{\phi}$ by
continuity. Thus if $Q \in \Prim(M(A))$ and $\overline{\phi}(Q)\in
E$ then $\overline \psi(Q)=e$. Hence $H_E \supseteq H_e$. For the
reverse inclusion, suppose that $Q\in{\rm Prim}(M(A))$ with $Q\notin
W:=\{ R\in {\rm Prim}(M(A)): R\supseteq H_E\}$. Let $D$ be a compact
neighbourhood of $Q$ in ${\rm Prim}(M(A))$ with $D$ disjoint from
$W$. Then $F:=\overline\phi(D)$ is a compact subset of $\beta
X_{\phi}$ and hence $L:=F\cap X_{\phi}$ is closed in $X_{\phi}$ and
disjoint from $E$, and $\overline\phi(Q)$ lies in the closure of $L$
in $\beta X_{\phi}$. Since $X_{\phi}$ is normal, $\overline\phi(Q)$
does not lie in the closure of $E$ in $\beta X_{\phi}$ and hence
there is a continuous, real-valued function $f$ on $\beta X_{\phi}$
such that $f(\overline\phi(Q))=1$ and $f(E)=\{0\}$. There is a
well-defined function $g: Y\to {\bf R}$ such that $g\circ
q=f|_{X_{\phi}}$. Let $U$ be an open subset of ${\bf R}$. Then
$q^{-1}(g^{-1}(U)) = f^{-1}(U) \cap X_{\phi}$, an open subset of
$X_{\phi}$. Thus $g^{-1}(U)$ is open in $Y$ and so $g$ is
continuous. Let $\overline g$ be the extension of $g$ to a
continuous function on $\beta Y$. Then $\overline g\circ \overline
q=f$ by continuity and so $\overline g\circ \overline q\circ
\overline\phi (Q)=1$. Hence $\overline{\psi}(Q) =\overline q\circ
\overline \phi(Q) \neq e$. Thus $Q\not \supseteq H_e$, and hence
$H_E=H_e$.
\end{proof}

%\bigskip

\section{Global and local spectral synthesis}

%\bigskip

 In this section we characterize `global spectral synthesis' by
showing that, for a $\sigma$-unital $C_0(X)$-algebra $A$, $H_x$ is
strictly closed for all $x\in X_{\phi}$ if and only if $J_x$ is
locally modular for all $x\in X_{\phi}$ and $\phi$ is a closed map
relative to its image (Corollary~\ref{Cor 2.7}). For separable $A$
we can also characterize `spectral synthesis at a point': if $A$ is
separable then for $x\in X_{\phi}$, $H_x$ is strictly closed if and
only if $J_x$ is locally modular and $\phi$ is locally closed at $x$
(Corollary~\ref{Cor 2.10}).

We begin by analyzing the property of $H_x$ being strictly closed
into two separate sub-properties (Proposition~\ref{Prop 2.2}). For
this, we need the following lemma.

%\bigskip

\begin{lemma} \label{Lemma 2.1} Let $A$ be a $\sigma$-unital
C$^*$-algebra and let $Y$ and $Z$ be disjoint closed subsets of
${\rm Prim}(A)$. Then the closures of $\tilde Y$ and $\tilde Z$ are
disjoint in ${\rm Prim}(M(A))$.
\end{lemma}

%\bigskip

\begin{proof} Set $J=\ker (Y\cup Z)$ and let $B=A/J$. Let $\pi: A\to
B$ be the quotient map. Then the map $P\mapsto\pi(P)$ $(P\in Y\cup
Z)$ carries $Y\cup Z$ homeomorphically onto ${\rm Prim}(B)$. Thus
there exists $b\in Z(M(B))$ such that
$b+\pi(P)^{\sim}=1+\pi(P)^{\sim}$ for $P\in Y$ and
$b+\pi(P)^{\sim}=0$ for $P\in Z$. By \cite[Theorem 10]{SAW} the
canonical map $M(A)\to M(B)$ is surjective, and hence there exists
$a\in M(A)$ such that $a+\tilde P=1+\tilde P$ for $P\in Y$ and
$a+\tilde P=0$ for $P\in Z$. Thus, by definition of the hull-kernel
topology, $a+Q=0$ for all $Q$ in the closure of $\tilde Z$ in ${\rm
Prim}(M(A))$, and by considering $(1-a)$ we see that likewise
$a+Q=1+Q$ for all $Q$ in the closure of $\tilde Y$ in ${\rm
Prim}(M(A))$. Hence $\tilde Y$ and $\tilde Z$ have disjoint closures
in ${\rm Prim}(M(A))$.
\end{proof}

%\bigskip

 The first of the two sub-properties into which we analyze the
property of strict closure is as follows. Recall that
$H(x)=\phi^{-1}(x)$ $(x\in X_{\phi})$. For $x\in X_{\phi}$ we say
that the base map $\phi$ is {\sl locally closed at $x$}
\cite[$\S$13.XIV]{Kur} if whenever $Y$ is a closed subset of ${\rm
Prim}(A)$ such that $x$ lies in the closure of $\phi(Y)$ then
$x\in\phi(Y)$, that is, $Y\cap H(x)$ is non-empty. For example, if
$x$ is an isolated point in $X_{\phi}$ (which implies that $H(x)$ is
a clopen subset of ${\rm Prim}(A)$) then $\phi$ is trivially locally
closed at $x$. Note, however, that $H(x)$ could be clopen, yet $x$
non-isolated in $X_{\phi}$. In this case, $\phi$ would not be
locally closed at $x$, see Example~\ref{Ex 3.8}(ii). We say that
$\phi$ is {\sl relatively closed} if $\phi(Y)$ is closed in
$X_{\phi}$ for all closed subsets $Y$ of ${\rm Prim}(A)$. Clearly
$\phi$ is relatively closed if and only if $\phi$ is locally closed
at each $x\in X_{\phi}$.

%\bigskip

\begin{prop} \label{Prop 2.2}  Let $A$ be a $C_0(X)$-algebra with
base map $\phi$ and let $x\in X_{\phi}$. Consider the following
properties:

{\rm (i)} $H_x$ is strictly closed;

{\rm (ii)} $\phi$ is locally closed at $x$;

{\rm (iii)} for each $b\in \tilde J_x$ and $\epsilon>0$ there is an
open set $V\subseteq {\rm Prim}(A)$ with $H(x)\subseteq V$ such that
$\Vert b+\tilde P\Vert<\epsilon$ for all $P\in V$.

\noindent Then {\rm (i)}$\Rightarrow${\rm (iii)} and {\rm
(ii)}$+${\rm (iii)}$\Rightarrow${\rm (i)}. If $A$ is $\sigma$-unital
then {\rm (i)}$\Rightarrow${\rm (ii)} and hence {\rm (i)} is
equivalent to {\rm (ii)}$+${\rm (iii)}.
\end{prop}

%\bigskip

\begin{proof} (i)$\Rightarrow$(iii) Suppose that $H_x=\tilde J_x$ and
let $b\in \tilde J_x$ and $\epsilon>0$. Then by the upper
semi-continuity of norm functions there is an neighbourhood $U$ of
$x$ in $X_{\phi}$ such that $\Vert b+H_y\Vert<\epsilon$ for all
$y\in U$. Set $V=\phi^{-1}(U)$. Then $H(x) \subseteq V$ and for all
$P\in V$
$$ \Vert b+\tilde{P}\Vert \leq \Vert b+\tilde{J}_{\phi(P)}\Vert
\leq \Vert b+H_{\phi(P)}\Vert <\epsilon.$$

(ii)$+$(iii)$\Rightarrow$(i) Let $b\in \tilde J_x$ and let
$\epsilon>0$ be given. Then by assumption there is an open set
$V\subseteq {\rm Prim}(A)$ with $H(x)\subseteq V$ such that $\Vert
b+\tilde P\Vert<\epsilon$ for all $P\in V$. Set $Y={\rm
Prim}(A)\setminus V$. Then $Y$ is closed and $Y\cap H(x)$ is empty.
Since $\phi$ is locally closed at $x$, it follows that $x$ does not
belong to the closure of $\phi(Y)$ in $X_{\phi}$. Hence there exists
$g\in C^b(X_{\phi})$ with $0\le g\le 1$ such that $g(\phi(Y))=\{1\}$
and $g(x)=0$. Let $z=\theta_A(g\circ\phi)\in Z(M(A))$ so that $0\leq
z\leq 1$ and $z+\tilde Q=g(\phi(Q))(1+\tilde Q)$ for all $Q\in {\rm
Prim}(A)$. Let $R\in {\rm Prim}(M(A))$ with $\overline{\phi}(R)=x$.
There is a net $(P_{\alpha})$ in ${\rm Prim}(A)$ such that
$\tilde{P_{\alpha}}\to R$. Then $x=\overline{\phi}(R)=
\lim\phi(P_{\alpha})$ and so $g(\phi(P_{\alpha})) \to g(x)=0$. Thus
$\Vert z+\tilde P_{\alpha}\Vert\to 0$ and so, since $z\in Z(M(A))$,
$\Vert z+R\Vert=0$. It follows that $z\in H_x$ and hence $zb\in
H_x$. But
$$\Vert zb-b\Vert=\sup\{\Vert (zb-b)+\tilde Q\Vert: Q\in {\rm
Prim}(A)\}=\sup\{\Vert (zb-b)+\tilde Q\Vert: Q\in V\}<\epsilon.$$
Since $\epsilon$ was arbitrary, $b\in H_x$ and hence $\tilde
J_x=H_x$.

(i)$\Rightarrow$(ii) (assuming that $A$ is $\sigma$-unital). Suppose
that $\phi$ is not locally closed at $x$ and let $Y$ be a closed
subset of ${\rm Prim}(A)$ such that $x$ lies in the closure of
$\phi(Y)$ but $H(x)\cap Y$ is empty. Let $W$ be the closure of
$\tilde Y$ in ${\rm Prim}(M(A))$. Then $\overline{\phi}(W)$ is a
compact and hence closed subset of $\beta X$ containing $\phi(Y)$.
Hence there exists $R\in W$ such that $\overline\phi(R)=x$ and
therefore $R\supseteq H_x$. By Proposition~\ref{Prop 1.1}(iv), the
closure of $\tilde H(x)=\{\tilde{P}: P\in \phi^{-1}(x)\}$ in ${\rm
Prim}(M(A))$ is equal to ${\rm Prim}(M(A)/\tilde J_x)$ (where the
latter is identified with the hull of $\tilde{J_x}$ in
$\Prim(M(A))$). But, since $A$ is $\sigma$-unital, $W$ is disjoint
from ${\rm Prim}(M(A)/\tilde J_x)$ by Lemma~\ref{Lemma 2.1} Thus
$R\not\supseteq \tilde J_x$ and so $\tilde J_x\ne H_x$.
\end{proof}

%\bigskip

\noindent In particular, we notice that if $A$ is $\sigma$-unital then
a necessary condition for spectral synthesis at $x\in X_{\phi}$
is that $\phi$ should be locally closed at $x$.

\bigskip

 To understand the second sub-property (property (iii) of
Proposition~\ref{Prop 2.2}) into which we have analyzed the property
of being strictly closed, we introduce the idea of local modularity.
To define this it is helpful to have the following notation. For
$x\in X_{\phi}$, let $\partial H(x)$ be the boundary and $U(x)$ the
interior of $H(x)$ in ${\rm Prim}(A)$. We say that $J_x$ is {\sl
locally modular} if for each $P\in \partial H(x)$ there exists a
relatively open neighbourhood $V$ of $P$ in ${\rm Prim}(A)\setminus
U(x)$ such that $A/\ker V$ is a unital C$^*$-algebra.

For instance, if $x$ is an isolated point of $X_{\phi}$ then $H(x)$
is clopen in ${\rm Prim}(A)$ and so $\partial H(x)$ is empty and
hence $J_x$ is vacuously locally modular. Secondly, if
$J_x\not\supseteq Z'(A)$ (that is, $x\in U_{\phi}$ ) then by upper
semi-continuity of norm functions and functional calculus we may
find $z\in Z'(A)$ such that $z+J_y= 1_{A/J_y}$ for all $y$ in an
open neighbourhood $V$ of $x$ in $X_{\phi}$ (cf. the proof of
\cite[Proposition 2.3]{Xalg}). Hence $z+P=1_{A/P}$ for all $P\in
W=\phi^{-1}(V)$, and $H(x)\subseteq W$, so again $J_x$ is locally
modular. On the other hand, if there exists $P\in \partial H(x)$
such that $A/P$ is non-unital then clearly $J_x$ is not locally
modular.

The definition of local modularity is intrinsic to $A$, in that it
does not mention $M(A)$, and it is a condition that should be
possible to check in concrete cases. The following equivalent
condition, however, seems easier to work with, although it does
involve $M(A)$. To describe this, we need a slight variant of the
definition of $\sim$. Recall from \cite{Som} that for $P, Q\in
\Prim(A)$ we write $P\sim Q$ if $P$ and $Q$ cannot be separated by
disjoint open subsets of $\Prim(A)$ (for a fuller discussion, see
Section 5 below). For the multiplier algebra $M(A)$ of a
$C_0(X)$-algebra $A$, we define $\sim_x$ as follows. For $x\in
X_{\phi}$ and $Q, R\in {\rm Prim}(M(A))\setminus \tilde U(x)$ we say
that $Q\sim_x R$ if there is a net $(P_{\alpha})$ in ${\rm
Prim}(A)\setminus U(x)$ such that $(\tilde P_{\alpha})$ converges to
both $Q$ and $R$. If $H(x)$ has empty interior then $\sim_x$
coincides with the relation $\sim$ on $\Prim(M(A))$ because the
canonical image of ${\rm Prim}(A)$ is dense in ${\rm Prim}(M(A))$.
Otherwise $Q\sim_x R\Rightarrow Q\sim R$, but the converse need not
hold.

For the next lemma, we need the definition of a primal ideal and of
the topology $\tau_s$. An ideal $J$ in a C$^*$-algebra $A$ is {\sl
primal} if whenever $I_1, \ldots, I_n$ is a finite collection of
ideals of $A$ with the product $I_1\ldots I_n=\{0\}$ then
$I_i\subseteq J$ for at least one $i\in \{ 1,\ldots , n\}$
\cite{AB}. Every primitive ideal is prime and hence primal. The set
of proper primal ideals of $A$ is denoted ${\rm Primal}'(A)$. The
$\tau_s$ topology on ${\rm Primal}'(A)$ is defined to be the weakest
topology for which all the norm functions $I\to \Vert a+I\Vert$
$(a\in A,\ I\in {\rm Primal}'(A))$ are continuous (see \cite[Section
II]{Fel}). If $A$ is unital then ${\rm Primal}'(A)$ is
$\tau_s$-compact \cite[Proposition 4.1]{Rob}.

%\bigskip

\begin{lemma} \label{Lemma 2.3}  Let $A$ be a $C_0(X)$-algebra with base
map $\phi$ and let $x\in X_{\phi}$. Consider the following
conditions:

{\rm (i)} $J_x$ is locally modular;

{\rm (ii)} for all $P\in \partial H(x)$ and $R\in {\rm
Prim}(M(A)/A)$, $\tilde P\not\sim_x R$.

\noindent Then {\rm (i)}$\Rightarrow${\rm (ii)}, and {\rm (i)} and
{\rm (ii)} are equivalent if $A$ is $\sigma$-unital.
\end{lemma}

%\bigskip

\begin{proof} Suppose first that (i) holds, and let $P\in
\partial H(x)$. Let $V$ be an open neighbourhood of $P$ in ${\rm
Prim}(A)\setminus U(x)$ such that $A/\ker V$ is unital. Let $W$ be
the closure of $V$ in ${\rm Prim}(A)$, so that $\ker V=\ker W$.
Write $J=\ker W$ (so that $W= {\rm Prim}(A/J)$) and recall that the
quotient map $q_J: A\to A/J$ has a canonical extension $\tilde q_J:
M(A)\to M(A/J)=A/J$. For each $b\in M(A)$ there exists $a\in A$ such
that $b-a\in\ker \tilde q_J=\tilde J$, so $A+\tilde J=M(A)$. Thus if
$R\in {\rm Prim}(M(A)/A)$ then $R\not\supseteq\tilde J$. It follows
that $\tilde W= {\rm Prim}(M(A)/\tilde J)$, a closed subset of ${\rm
Prim}(M(A))$. Hence if $(P_{\alpha})$ is a net in ${\rm
Prim}(A)\setminus U(x)$ with $P_{\alpha}\to P$ then eventually
$P_{\alpha}\in V$, so all the limits of the net $(\tilde
P_{\alpha})$ in ${\rm Prim}(M(A))$ lie in $\tilde W$. Thus $\tilde
P\not\sim_x R$ for any $R\in {\rm Prim}(M(A)/A)$.

Conversely, suppose that (ii) holds and that $A$ is $\sigma$-unital. Let
$u$ be a strictly positive element in $A$ with $\Vert u\Vert=1$. Let $Q\in \partial H(x)$
and suppose that $A+\tilde Q\ne M(A)$.
Then there exists a maximal ideal $M$ of $M(A)$
with $M\supseteq A+\tilde Q$, and hence $M\sim_x \tilde Q$ contradicting (ii).
Thus $A+\tilde Q=M(A)$, so $A/Q$ is unital
and $\Vert (1-u)+\tilde Q\Vert<1$.

Let $P\in \partial H(x)$ and suppose, for a contradiction, that
there is a net $(P_{\alpha})$ in ${\rm Prim}(A)\setminus U(x)$ with
$P_{\alpha}\to P$ and $\Vert (1-u)+\tilde P_{\alpha} \Vert\to 1$. By
the $\tau_s$-compactness of ${\rm Primal}'(M(A))$, and by passing to
a subnet if necessary, we may assume that there exists $J\in {\rm
Primal}'(M(A))$ such that $\tilde P_{\alpha}\to J$ ($\tau_s$). Hence
$\Vert (1-u)+J\Vert=1$ and so there exists $R\in {\rm Prim}(M(A)/J)$
such that $\Vert (1-u)+R\Vert=1$ \cite[3.3.6]{Dix}. Since $\tilde
P_{\alpha}\to J$ $(\tau_s)$, $\tilde P_{\alpha}\to R$ and so
$R\not\supseteq A$ by (ii). Then $R=\tilde Q$, where $Q:=R\cap A\in
{\rm Prim}(A)$, and so $P_{\alpha}\to Q$. Since $\phi$ is
continuous, $\phi(Q)=\phi(P)=x$ and so $Q\in H(x)$. But
$P_{\alpha}\in {\rm Prim}(A)\setminus U(x)$ for all $\alpha$ and so
$Q\in \partial H(x)$, contradicting the fact that $\Vert
(1-u)+\tilde Q\Vert=1$. Thus no such net $(P_{\alpha})$ exists and
so there exists $\epsilon>0$ and a neighbourhood $V$ of $P$ in ${\rm
Prim}(A)\setminus U(x)$ such that $\Vert (1-u)+\tilde
Q\Vert<1-\epsilon$ for all $Q\in V$.

Now let $f$ be a continuous function on $[0,1]$ with $f(0)=0$,
$f([\epsilon, 1])=\{ 1\}$, and $\Vert f\Vert=1$, and set $v=f(u)$.
Then $(1-v)+\tilde Q=0$ for all $Q\in V$ and so $v+\ker V$ is the
identity in $A/\ker V$. Hence $J_x$ is locally modular, as required.
\end{proof}

%\bigskip

 The next proposition shows part of the connection between local
modularity and property (iii) of Proposition~\ref{Prop 2.2}.

%\bigskip

\begin{prop} \label{Prop 2.4} Let $A$ be a $C_0(X)$-algebra with
base map $\phi$ and let $x\in X_{\phi}$. If $J_x$ is locally modular
then property (iii) of Proposition~\ref{Prop 2.2} holds, namely for
each $b\in \tilde J_x$ and $\epsilon>0$ there is an open set
$V\subseteq {\rm Prim}(A)$ with $H(x)\subseteq V$ such that $\Vert
b+\tilde P\Vert<\epsilon$ for all $P\in V$. Hence if $J_x$ is
locally modular and $\phi$ is locally closed at $x$ then $H_x$ is
strictly closed.
\end{prop}

%\bigskip

\begin{proof} Suppose that $J_x$ is locally modular and let $b\in
\tilde{J}_x$ and $\epsilon >0$. Let $P\in \partial H(x)$ and suppose
for a contradiction that there is a net $(P_{\alpha})$ in ${\rm
Prim}(A)\setminus U(x)$ such that $P_{\alpha}\to P$ and $\Vert
b+\tilde P_{\alpha}\Vert\ge\epsilon$. Then by compactness of the set
$W=\{R\in {\rm Prim}(M(A)): \Vert b+R\Vert\ge\epsilon\}$, and by
passing to a subnet of $(P_{\alpha})$ if necessary, there exists
$R\in W$ such that $\tilde P_{\alpha}\to R$. Since $J_x$ is locally
modular, it follows from Lemma~\ref{Lemma 2.3} that $R=\tilde{Q}$
for some $Q\in {\rm Prim}(A)$. Since $P_{\alpha}\to Q$,
$\phi(Q)=\phi(P)=x$ and so $b\in \tilde{Q}$, contradicting the fact
that $\Vert b+\tilde{Q}\Vert\geq\epsilon$. Therefore no such net
$(P_{\alpha})$ exists. Thus there is an open set $V_P$ containing
$P$ such that $\Vert b+\tilde Q\Vert<\epsilon$ for all $Q\in V_P$.
Taking $V=U(x)\cup \bigcup_{P\in \partial H(x)} V_P$ gives the
required set $V$.

The final statement now follows from Proposition~\ref{Prop 2.2}
((ii)$+$(iii)$\Rightarrow$(i)).
\end{proof}

%\bigskip

 It will follow from Theorem~\ref{Thm 2.6} that the converse to
 Proposition~\ref{Prop 2.4} holds if $A$ is separable. For a general $\sigma$-unital
$C_0(X)$-algebra $A$, however, it is possible for property (iii) of
Proposition~\ref{Prop 2.2} to hold for a particular $x\in X_{\phi}$
without $J_x$ being locally modular, see Example~\ref{Ex 3.5}.
Nevertheless the relation between the two properties is very close,
as we shall see.

To show this, we need the following theorem from \cite[Theorem
2.5]{multid}. For an element $a$ in a C$^*$-algebra $A$, let ${\rm
sp}(a)$ denote the spectrum of $a$; and for $a\ge 0$, let $\min {\rm
sp}(a)$ be the smallest number in ${\rm sp}(a)$. The function $g$
from the unit interval $[0,1]$ to the space $C[0,1]$ is as follows
(where for $r\in [0,1]$, $g_r$ is the continuous function on $[0,1]$
corresponding to $r$):

$$g_0(t)=1 \hbox{ for all } t\in [0,1];$$

$$\hbox{ for }0<r\le 1/2,\ \ \ g_r(t)=
\begin{cases}0    &(0\le t\le r/2)\cr
 (2t/r) -1 &(r/2\le t\le r)\cr
 1   &(r\le t\le 1);
 \end{cases}$$
$$g_r=g_{1/2} \hbox{ for }r\ge 1/2.$$

%$$g_0(x)=1 \hbox{ for all } x\in [0,1];$$

%$$\hbox{ for }0<r\le 1/2\ \ \ g_r(x)=\cases{0 & $0\le x\le r/2$\cr
%(2x/r) -1 & $r/2\le x\le r$\cr 1 & $r\le x\le 1$;\cr }$$
%$$g_r=g_{1/2} \hbox{ for }r\ge 1/2.$$

%\bigskip

\begin{thm} \label{Thm 2.5}  Let $A$ be a $\sigma$-unital
$C_0(X)$-algebra with base map $\phi$ and set $X_{\phi}={\rm
Im}(\phi)$.
 Let $u$ be a strictly
positive element in $A$ with $\Vert u\Vert=1$. Let $f\in C^b(X_{\phi})$
with $0\leq f\leq 1$, let $U$ be the cozero set of $f$ and let $V=\{
x\in U: 2 \min {\rm sp}(u+J_x)\le f(x)\}$. Let
${\rm cl}(U)$ and ${\rm cl}(V)$ be the closures of $U$ and $V$ respectively in
$X_{\phi}$. Then there exists $b\in M(A)$ with $0\le b\le 1$ such that

 {\rm (i)} $b+ \tilde{J_x}=g_{f(x)}(u+\tilde{J_x})\ \ \ \ \ \ \ \ (x\in X_{\phi})$;

 {\rm (ii)} $b\in A+H_x\subseteq A+\tilde{J_x}$ for all $x\in U$;

 {\rm (iii)} $1-b\in \tilde{J_x}$ for all $x\in X_{\phi}\setminus U$ and $1-b\in H_x$
 for all $x\in X_{\phi}\setminus{\rm cl}(U)$;

 {\rm (iv)}  $\Vert (1-b)+\tilde{J_x}\Vert=1$ for all $x\in V$ and $\Vert
(1-b)+H_x\Vert= 1$ for all $x\in {\rm cl}(V)$.

\noindent Furthermore,

 {\rm (v)} $H_x$ is not strictly closed in $M(A)$ for all $x\in
{\rm cl}(V)\setminus U$.
\end{thm}

%\bigskip

\noindent In the context of Theorem~\ref{Thm 2.5}, note that if
$x\in U$ and $A/J_x$ is non-unital then $0\in{\rm sp}(u+J_x)$ and
hence $x\in V$.

%\bigskip

\begin{thm} \label{Thm 2.6}  Let $A$ be a $\sigma$-unital
$C_0(X)$-algebra with base map $\phi$. Let $x\in X_{\phi}$ and let
$Z$ be a zero set of $X_{\phi}$ with $x\in Z$. Suppose that $J_x$ is
not locally modular. Then there exists $y\in Z$ for which property
(iii) of Proposition~\ref{Prop 2.2} fails, that is, for which there
exists $c\in \tilde J_y$ and $\epsilon>0$ such that there is no open
set $V\subseteq {\rm Prim}(A)$ with $H(y)\subseteq V$ and $\Vert
c+\tilde P\Vert<\epsilon$ for all $P\in V$. In particular, $H_y$ is
not strictly closed.
\end{thm}

%\bigskip

\begin{proof} Since $J_x$ is not locally modular, it follows from
Lemma~\ref{Lemma 2.3} that there exist $P\in \partial H(x)$ and
$R\in {\rm Prim}(M(A)/A)$ such that $\tilde P\sim_x R$. Let $D$ be a
compact neighbourhood of $P$ in ${\rm Prim}(A)$. Let $u\in A$ be a
strictly positive element with $\Vert u\Vert=1$. Since $u\in A$, it
follows that $\Vert (1-u)+R\Vert=1$. Hence by \cite[3.3.2]{Dix}, for
each $n$ the set
$$U_n=\{ Q\in {\rm Prim}(M(A)): \Vert (1-u)+ Q\Vert>1-1/2^{n+2}\}$$
is an open neighbourhood of $R$ in ${\rm Prim}(M(A))$. Let $h:X_{\phi}\to [0,1]$ be a continuous function such that $Z(h)=Z$ and,
for $n\ge 1$, let $V_n=\{ Q\in {\rm Prim}(A): h(\phi(Q))<1/n\}$.
Then $V_n$ is an open neighbourhood of $P$ in ${\rm Prim}(A)$. Since
$\tilde P\sim_x R$, it follows that
$$U_1\cap \tilde V_1\cap ({\rm Int}\,D)^{\sim}\cap ({\rm Prim}(A)\setminus U(x))^{\sim}$$
 is a non-empty,
relatively open, subset of $({\rm Prim}(A)\setminus U(x))^{\sim}$.

Since $\Prim(A)\setminus H(x)$ is dense in $\Prim(A)\setminus U(x)$,
we may choose $Q_1\in {\rm Prim}(A)\setminus H(x)$ with $\tilde
Q_1\in U_1\cap \tilde V_1\cap ({\rm Int}\,D)^{\sim}$. Set $x_1=\phi(Q_1)$. Then
$x_1\ne x$ and so there exists a continuous function
$f_1:X_{\phi}\to [0,1/4]$ with $f_1(x_1)=1/4$ and $f_1(x)=0$. For
$n\ge 2$, we will inductively define points $x_n\in X_{\phi}$ and
continuous functions $f_n:X_{\phi}\to [0, 1/2^{n+1}]$ with
$f_n(x_n)=1/2^{n+1}$ and $f_n(x)=0=f_n(x_m)$ for $1\le m\le n-1$.
Note that $x_1$ and $f_1$ satisfy these conditions.

Suppose that $n\geq2$ and that $x_1,\ldots, x_{n-1}$ and
$f_1,\ldots,f_{n-1}$ satisfy the required conditions. Let $W_n=\{
Q\in {\rm Prim}(A): \sum_{i=1}^{n-1} f_{i}(\phi(Q))<1/2^{n+1}\}$.
Then $W_n$ is an open neighbourhood of $P$ and hence, since $\tilde
P\sim_x R$, it follows that
$$U_n\cap \tilde V_n\cap({\rm Int}\,D)^{\sim}\cap
\tilde W_n\cap ({\rm Prim}(A)\setminus U(x))^{\sim}$$
 is a non-empty,
relatively open, subset of $({\rm Prim}(A)\setminus U(x))^{\sim}$.
Thus we may choose $Q_n\in {\rm Prim}(A)\setminus H(x)$ with
$\tilde{Q}_n\in U_n\cap \tilde V_n\cap({\rm Int}\,D)^{\sim}\cap\tilde W_n$.
%Note that $\Vert
%(1-u)+\tilde Q_n\Vert>1-1/2^{n+2}$ for $n\ge 1$.
Set
$x_n=\phi(Q_n)$. Then $x_n\ne x$ and $x_n\ne x_m$ for $1\le m\le
n-1$ because
$$\sum_{i=1}^{n-1} f_{i}(x_m)\ge f_m(x_m)=\frac{1}{2^{m+1}}>
\frac{1}{2^{n+1}}> \sum_{i=1}^{n-1} f_{i}(x_n).$$
Thus there exists a continuous function $f_n:X_{\phi}\to [0,
1/2^{n+1}]$ with $f_n(x_n)=1/2^{n+1}$ and $f_n(x)=0=f_n(x_m)$ for
$1\le m\le n-1$.

Set $f=\sum_{n=1}^{\infty} f_n$. Then $f$ is continuous and, for each $n\geq1$,
$$\frac{1}{2^{n+1}}=f_n(x_n)\leq f(x_n)= \sum_{i=1}^{n-1} f_{i}(x_n)+f_n(x_n)<\frac{1}{2^{n+1}}+\frac{1}{2^{n+1}}=\frac{1}{2^n}.$$
Thus $f(x_n)\to0$ as $n\to\infty$.
Furthermore,
since $\Vert
(1-u)+\tilde{Q}_n\Vert>1-1/2^{n+2}$, it follows that
$$\min {\rm
sp}(u+\tilde{Q}_n) <\frac{1}{2^{n+2}}\le \frac{f(x_n)}{2}$$
 and hence $0\in{\rm sp}(
g_{f(x_n)}(u+\tilde{Q}_n) )$
   for $n\geq 1$.

Since $Q_n\in D$ for all $n$, the compactness of $D$ implies that
there exists $Q\in D$ and a subnet $(Q_{n_{\alpha}})$ of $(Q_n)$
such that $Q_{n_{\alpha}}\to Q$. Set $y=\phi(Q)$. Since $Q_n\in V_n$ ($n\geq1$), $h(x_n)\to
0$ as $n\to
\infty$.  It follows that $h(y)=0$ and hence $y\in Z$. Furthermore,
$f(x_{n_{\alpha}})\to f(y)$ and so $f(y)=0$. Let $b$ be an element of
$M(A)$ corresponding to $f$ as in Theorem~\ref{Thm 2.5}. Then
$1-b\in \tilde J_y$ by Theorem~\ref{Thm 2.5}(iii).  Since
$\phi(Q_n)=x_n$ ($n\geq1$), it follows that $\tilde{Q}_n\supseteq
\tilde{J}_{x_n}$ and hence that $b+
\tilde{Q}_n=g_{f(x_n)}(u+\tilde{Q}_n)$ by Theorem~\ref{Thm 2.5}(i).
Then
$$1 \ge\Vert(1-b) + \tilde{Q}_n\Vert \ge \Vert(1+\tilde{Q}_n) -g_{f(x_n)}(u+\tilde{Q}_n)
\Vert
= 1-\min {\rm sp}( g_{f(x_n)}(u+\tilde{Q}_n) )  =1.$$ Hence $\Vert
(1-b)+\tilde Q_{n_{\alpha}}\Vert=1$ for all $\alpha$. Suppose that
$V$ is an open subset of $\Prim(A)$ such that $H(y) \subseteq V$.
Then $Q\in V$ and so eventually $Q_{n_{\alpha}} \in V$. Thus
property (iii) of Proposition~\ref{Prop 2.2} fails at $y$ for $c =
1-b \in \tilde J_y$ and $\epsilon=1/2$. Hence $H_y$ is not strictly
closed by Proposition~\ref{Prop 2.2}.
\end{proof}

%\bigskip

\noindent Regarding the hypotheses of Theorem~\ref{Thm 2.6}, we note
that the complete regularity of $X_{\phi}$ implies that every
neighbourhood of $x$ contains a zero set containing $x$. In the
proof of Theorem~\ref{Thm 2.6}, we could have checked that $y \in
{\rm cl}(V)\setminus U$ (in the terminology of Theorem~\ref{Thm
2.5}) and used Theorem~\ref{Thm 2.5}(v) to deduce that $H_y$ is not
strictly closed. We have preferred, however, to obtain the stronger
result that property (iii) of Proposition~\ref{Prop 2.2} fails at
$y$.

\bigskip

 Armed with Theorem~\ref{Thm 2.6}, we can now prove some of the main results
of the paper.

%\bigskip

\begin{cor}[{\bf global spectral synthesis}] \label{Cor 2.7}  Let $A$
be a $C_0(X)$-algebra with base map $\phi$. Consider the following
two properties.

{\rm (i)} $H_x$ is strictly closed for all $x\in X_{\phi}$.

{\rm (ii)} $J_x$ is locally modular and $\phi$ is locally closed at
$x$ for all $x\in X_{\phi}$.

\noindent Then {\rm (ii)}$\Rightarrow${\rm (i)}, and {\rm (i)} and
{\rm (ii)} are equivalent if $A$ is $\sigma$-unital.
\end{cor}

%\bigskip

\begin{proof} (ii)$\Rightarrow$(i). This follows from Proposition~\ref{Prop 2.4}.

(i)$\Rightarrow$(ii) (assuming that $A$ is $\sigma$-unital). This
follows from (i)$\Rightarrow$((ii)$+$(iii)) of Proposition~\ref{Prop
2.2} together with Theorem~\ref{Thm 2.6} (taking $Z=X_{\phi}$, for
example).
\end{proof}

%\bigskip

\begin{cor}[{\bf spectral synthesis at a point}] \label{Cor 2.8}  Let
$A$ be a $\sigma$-unital $C_0(X)$-algebra with base map $\phi$ and
let $x$ be a $G_{\delta}$-point in $X_{\phi}$. Then $H_x$ is
strictly closed if and only if $J_x$ is locally modular and $\phi$
is locally closed at $x$.
\end{cor}

%\bigskip

\begin{proof} This follows from Proposition~\ref{Prop 2.4} and from
Proposition~\ref{Prop 2.2} and Theorem~\ref{Thm 2.6} taking the zero
set $Z$ in Theorem~\ref{Thm 2.6} to be $Z=\{x\}$.
\end{proof}

%\bigskip

 Recall that a completely regular (Hausdorff) space $X$ is {\sl
perfectly normal} if every closed subset of $X$ is the zero set of a
continuous real-valued function on $X$. Every metric space is
perfectly normal.

%\bigskip

\begin{lemma} \label{Lemma 2.9} Let $A$ be a separable $C_0(X)$-algebra
with base map $\phi$. Then $X_{\phi}$ is perfectly normal.
\end{lemma}

%\bigskip

\begin{proof}   Let $V$ be an open subset of $X_{\phi}$. Set
$Y=\phi^{-1}(V)$ and let $J$ be the ideal of $A$ such that ${\rm
Prim}(J)$ is canonically homeomorphic to $Y$. Then $J$ is separable.
Let $u$ be a strictly positive element in $J$. For each $n\in {\bf
N}$, set $Y_n=\{P\in {\rm Prim}(A): \Vert u+P\Vert\ge 1/n\}$. Then
$Y_n$ is compact and so  $\phi(Y_n)$ is a compact (hence closed)
subset of $X_{\phi}$. Since $V=\bigcup_{n=1}^{\infty}\phi(Y_n)$, we
see that $V$ is an $F_{\sigma}$ subset of $X_{\phi}$. But $X_{\phi}$
is normal, and in a normal space an open $F_{\sigma}$ subset is a
cozero set. Thus $V$ is a cozero set.
\end{proof}

%\bigskip

\begin{cor} \label{Cor 2.10}  Let $A$ be a separable
$C_0(X)$-algebra with base map $\phi$ and let $x\in X_{\phi}$. Then
$H_x$ is strictly closed if and only if $J_x$ is locally modular and
$\phi$ is locally closed at $x$.
\end{cor}

%\bigskip

\begin{proof} This follows as for Corollary~\ref{Cor 2.8}, noting that $\{x\}$
is a zero set of $X_{\phi}$ for each $x\in X_{\phi}$ by
Lemma~\ref{Lemma 2.9}.
\end{proof}

%bigskip

%\bigskip

\section{Pointwise spectral synthesis and P-points}

%\bigskip

 Recall from Section 2 that if $A$ is a $C_0(X)$-algebra with base
map $\phi$ then $U_{\phi}=\{x\in X_{\phi}: J_x\not\supseteq
Z'(A)\}$, an open subset of $X_{\phi}$, $\partial U_{\phi}$ denotes
the boundary of $U_{\phi}$ in $X_{\phi}$ and
$W_{\phi}=X_{\phi}\setminus U_{\phi}$. Here $Z'(A)=\mu(C_0(X))\cap
A$. We saw in Proposition~\ref{Prop 1.7} that if $x\in U_{\phi}$ or
if $x$ is an isolated point of $X_{\phi}$ then $H_x$ is strictly
closed, while if $A/J_x$ is unital but $x\notin U_{\phi}$ then $H_x$
is not strictly closed. We remarked that the remaining points to
consider are those for which $A/J_x$ is non-unital and $x\in
W_{\phi}$ and is non-isolated in $X_{\phi}$. From Section 3 we now
have a complete characterization of global spectral synthesis for
$\sigma$-unital $A$, and a complete characterization of pointwise
spectral synthesis for separable $A$, but only a
near-characterization of pointwise spectral synthesis for the more
general case when $A$ is $\sigma$-unital.

In this section we approach the $\sigma$-unital case from another
angle. We characterize the points $x$ in the interior of $W_{\phi}$
in $X_{\phi}$ for which $H_x$ is closed (these turn out to be
precisely the P-points) and make partial progress for the difficult
case of points $x\in\partial U_{\phi}$ (cf. for example
\cite[Proposition 3.5]{Xalg}). We also give a complete
characterization of pointwise spectral synthesis for the important
case when the base map $\phi$ is open (Theorem~\ref{Thm 3.6}). We
begin with the following lemma.

%\bigskip

\begin{lemma} \label{Lemma 3.1}  Let $A$ be a $\sigma$-unital
$C_0(X)$-algebra with base map $\phi$ and let $x\in W_{\phi}$. Let
$u\in A$ be strictly positive with $\Vert u\Vert\le 1$. Then there
is a net $(x_{\alpha})$ in $X_{\phi}$ with $x_{\alpha}\to x$ and
$\Vert (1-u)+\tilde J_{x_{\alpha}}\Vert\to 1$.
\end{lemma}

%\bigskip

\begin{proof} By Lemma~\ref{Lemma 1.6} there exists $R\in \Prim(M(A)/A)$ with
$R\supseteq H_x$. Let $(P_{\alpha})$ be a net in ${\rm Prim}(A)$
such that $\tilde P_{\alpha}\to R$. Then $\overline\phi(\tilde
P_{\alpha}) \to\overline\phi(R)=x$. Hence
$x_{\alpha}:=\phi(P_{\alpha})\to x$. On the other hand, since
$R\supseteq A$,
\begin{align*}1=\Vert (1-u)+R\Vert
&\le\lim\inf\Vert (1-u)+\tilde P_{\alpha}\Vert \\
&\le\lim\inf\Vert (1-u)+\tilde J_{x_{\alpha}}\Vert\\
&\le\lim\sup\Vert (1-u)+\tilde J_{x_{\alpha}}\Vert\le 1.
\end{align*}
 Hence $\Vert (1-u)+\tilde J_{x_{\alpha}}\Vert\to 1$.
 \end{proof}

%\bigskip

 Next we recall the definition of a P-point. Let $X$ be a
completely regular (Hausdorff) space. A point $x\in X$ is a {\sl
P-point} if every continuous real-valued function vanishing at $x$
vanishes in a neighbourhood of $x$ \cite[4L]{GJ}. Equivalently, $x$
is a P-point if $x$ does not lie in the boundary of any cozero set.
If the space $X$ is perfectly normal then every singleton is a zero
set and so a P-point is necessarily an isolated point. A space in
which every point is a P-point is a {\sl P-space}.

We are now ready for the first main result of this section.

%\bigskip

\begin{thm} \label{Thm 3.2}  Let $A$ be a $\sigma$-unital
$C_0(X)$-algebra with base map $\phi$. If $x\in X_{\phi}$ is a
P-point in $X_{\phi}$ then $H_x$ is strictly closed. Conversely, if
$x\in W_{\phi}$ and $H_x$ is strictly closed then $x$ is a P-point
in $W_{\phi}$.
\end{thm}

%\bigskip

\begin{proof}  Let $x\in X_{\phi}$ and suppose that $H_x$ is not
strictly closed. We show that $x$ is not a P-point in $X_{\phi}$. By
\cite[Theorem 10.1.7]{KR}, there exists $b\in \tilde J_x\setminus
H_x$ with $\Vert b\Vert=\Vert b+H_x\Vert=1$.  Let $u$ be a strictly
positive element of $A$ with $\Vert u\Vert=1$, and recall that for
$y\in X_{\phi}$, $b\in \tilde J_y$ if and only if $bu\in\tilde J_y$
(cf. \cite[Section 2]{Xalg}). For each $n\ge 1$, set $W_n=\{ y\in
X_{\phi}: \Vert bu+\tilde J_y\Vert\ge 1/n\}$. By \cite[Lemma
4.2]{multid}, for every neighbourhood $W$ of $x$ in ${\rm cl}_{\beta
X}X_{\phi}$,
$$ ||b+H_x|| \leq \sup\{ ||b+\tilde{J_y}||: y\in X_{\phi}\cap W\}.$$
Hence for every neighbourhood $V$ of $x$ in $X_{\phi}$ there exists
$y\in V$ such that $bu\notin \tilde J_y$. Thus $x$ lies in the
closure of $\bigcup_{n=1}^{\infty} W_n$. Since $$ W_n = \{y\in
X_{\phi}: ||bu+J_y|| \geq 1/n\} =\phi(\{P\in {\rm Prim}(A): ||bu+P||
\geq 1/n\}),$$ $W_n$ is a compact subset of $X_{\phi}$. But $x\notin
W_n$ and hence there exists a continuous function $f_n: X_{\phi}\to
[0,1]$ such that $f_n(x)=0$ and $f_n(W_n)=\{ 1\}$. Set
$f=\sum_{n=1}^{\infty} f_n/2^n$. Then $f(x)=0$, but $x$ lies in the
closure of the cozero set of $f$. Hence $x$ is not a P-point.

Conversely, suppose that $x\in W_{\phi}$ is not a P-point in
$W_{\phi}$. Let $f$ be a continuous function on $W_{\phi}$ with
$f(x)=0$ such that $x$ lies in the closure of the cozero set of $f$.
Replacing $f$ by $\min\{ |f|, 1\}$, we may assume that $0\le f\le
1$. Since $X_{\phi}$ is normal and $W_{\phi}$ is closed in
$X_{\phi}$, we may extend $f$ to a continuous function $\overline f$
on $X_{\phi}$ with $0\le \overline f\le 1$. Let $y\in W_{\phi}\cap
{\rm coz}(\overline f)$.  By Lemma~\ref{Lemma 3.1} there is a net
$(y_{\alpha})$ in $X_{\phi}$ with $y_{\alpha}\to y$ and $\Vert
(1-u)+\tilde J_{y_{\alpha}}\Vert\to 1$. Hence eventually $2\min {\rm
sp}(u +J_{y_{\alpha}})=2 (1-\Vert (1-u)+\tilde J_{y_{\alpha}}\Vert)
\le \overline f(y_{\alpha})$, since $\overline f(y_{\alpha})\to
\overline f(y)>0$. It follows that the set $V$ of Theorem~\ref{Thm
2.5} (associated with the cozero set ${\rm coz}(\overline f) $) has
closure ${\rm cl}(V)$ containing $W_{\phi}\cap {\rm coz}(\overline
f)$. Hence $x\in {\rm cl}(V)$ since $x$ lies in the the closure of
$W_{\phi}\cap {\rm coz}(\overline f)={\rm coz}(f)$. Thus $\tilde J_x
\neq H_x$ by Theorem~\ref{Thm 2.5}(v).
\end{proof}

%\bigskip

\noindent Theorem~\ref{Thm 3.2}
 has some useful consequences. Let
${\rm Int}\ W_{\phi}$ denote the interior of $W_{\phi}$ relative to
$X_{\phi}$.

%\bigskip

\begin{cor} \label{Cor 3.3}  Let $A$ be a
 $C_0(X)$-algebra with base map $\phi$ and let
$x\in {\rm Int}\ W_{\phi}$. If $A$ is $\sigma$-unital then $H_x$ is
strictly closed if and only if $x$ is a P-point in $X_{\phi}$. If
$A$ is separable then $H_x$ is strictly closed if and only if $x$ is
an isolated point in $X_{\phi}$.
\end{cor}

%\bigskip

\begin{proof} Suppose that $A$ is $\sigma$-unital and that $x$ is not
a P-point in $X_{\phi}$. Then there exists $f\in C^b(X_{\phi})$ such
that $f(x)=0$ and $x$ lies in the closure of $\coz(f)$. Let
$(x_{\alpha})$ be a net in $\coz(f)$ such that $x_{\alpha}\to x$,
and set $\overline f=f|_{W_{\phi}}$. Then $\overline f$ is
continuous and $\overline f(x)=0$ but eventually $x_{\alpha}\in {\rm
Int}\ W_{\phi}$ and so $x$ lies in the closure of the cozero set of
$\overline f$. Thus $x$ is not a P-point in $W_{\phi}$ and so $H_x$
is not strictly closed by Theorem~\ref{Thm 3.2}. If $A$ is separable
then $\{x\}$ is a zero set in $X_{\phi}$ by Lemma~\ref{Lemma 2.9}
and hence $x$ is P-point in $X_{\phi}$ if and only if it is isolated
in $X_{\phi}$.
\end{proof}

%\bigskip

\begin{cor} \label{Cor 3.4}  Let $A$ be a $C_0(X)$-algebra with
base map $\phi$ and suppose that $Z'(A)=\{0\}$.

{\rm (i)} Let $x\in X_{\phi}$. If $A$ is $\sigma$-unital then $H_x$
is strictly closed if and only if $x$ is a P-point in $X_{\phi}$. If
$A$ is separable then $H_x$ is strictly closed if and only if $x$ is
an isolated point in $X_{\phi}$.

{\rm (ii)} If $A$ is $\sigma$-unital then $H_x$ is strictly closed
for all $x\in X_{\phi}$ if and only if $X_{\phi}$ is discrete.
\end{cor}

%\bigskip

\begin{proof} Since $Z'(A)=\{0\}$,  $U_{\phi}=\emptyset$ and
$W_{\phi}= X_{\phi}$. Thus part (i) follows from Corollary~\ref{Cor
3.3}. Part (ii) now follows because the space $X_{\phi}$ is
$\sigma$-compact, and a $\sigma$-compact P-space is discrete (see, for example, the proof of  \cite[Lemma 4.4]{multid}).
\end{proof}

%\bigskip

\noindent If $A$ is a stable $C_0(X)$-algebra then $Z(A)=\{0\}$ and
hence $Z'(A) = \{0\}$. Thus Corollary~\ref{Cor 3.4} extends
\cite[Theorem 4.5]{multid}.

It follows from Theorem~\ref{Thm 3.2} and Proposition~\ref{Prop 2.2}
that if $A$ is a $\sigma$-unital $C_0(X)$-algebra and $x\in
X_{\phi}$ is a P-point in $X_{\phi}$ then $\phi$ is locally closed
at $x$ and that property (iii) of Proposition~\ref{Prop 2.2} holds.
On the other hand, $J_x$ need not be locally modular as the
following example shows.

%\bigskip

\begin{example} \label{Ex 3.5} {\rm Let $X$ be a compact Hausdorff space with a
non-isolated P-point $x$ (e.g. take $X$ to be $\omega_1+1$, where
$\omega_1$ is the first uncountable ordinal, with the usual
topology) and let $A=C(X)\otimes K(H)$ (where $K(H)$ is the algebra
of compact operators on a separable, infinite-dimensional Hilbert
space $H$). There is a homeomorphism $\phi:\Prim(A)\to X$ such that
$$\phi(\{f\in C_0(X):f(y)=0\}\otimes K(H)) = y \qquad (y\in X).$$
Then $J_x$ is not locally modular because $\partial H(x)$ is
non-empty but $A/P$ is non-unital for all $P\in\Prim(A)$.
Nevertheless, $H_x$ is strictly closed by Theorem~\ref{Thm 3.2}.}
\end{example}

%\bigskip

 In Theorem~\ref{Thm 3.2} we saw a characterization, for $A$ $\sigma$-unital,
of when $H_x$ is strictly closed for $x\in {\rm Int}\ W_{\phi}$ and
we know that $H_x$ is always strictly closed if $x\in U_{\phi}$
(Proposition~\ref{Prop 1.7}(i)). The remaining points to consider
are those in $\partial U_{\phi}$. For general $\sigma$-unital
$C_0(X)$-algebras we are not able to characterize the points
$x\in\partial U_{\phi}$ for which $H_x$ is strictly closed (though
we have seen a necessary condition in Theorem~\ref{Thm 3.2}), but if
we make the further assumption that the base map $\phi$ is open then
we can show that there are no such points.

%\bigskip

\begin{thm} \label{Thm 3.6}  Let $A$ be a continuous,
$\sigma$-unital $C_0(X)$-algebra with base map $\phi$ and let $x\in
X_{\phi}$. Then the following are equivalent:

{\rm (i)} $H_x$ is strictly closed;

{\rm (ii)} either $x\in U_{\phi}$, or $x\in {\rm Int}\ W_{\phi}$ and
$x$ is a P-point in $X_{\phi}$.
\end{thm}

%\bigskip

\begin{proof} (ii)$\Rightarrow$(i) If $x\in U_{\phi}$, or if $x\in
{\rm Int}\ W_{\phi}$ with $x$ a P-point in $X_{\phi}$, then $H_x$ is
strictly closed by Proposition~\ref{Prop 1.7}(i) and
Corollary~\ref{Cor 3.3}.

(i)$\Rightarrow$(ii) Corollary~\ref{Cor 3.3} shows that if $x\in
{\rm Int}\ W_{\phi}$ with $H_x$ strictly closed then $x$ must be a
P-point in $X_{\phi}$. It is enough, therefore, to show that if
$x\in\partial U_{\phi}$ then $H_x$ is not strictly closed. If
$A/J_x$ is unital then this follows from Proposition~\ref{Prop
1.7}(iii). Hence we may assume that $A/J_x$ is non-unital. Let $u$
be a strictly positive element of $A$ with $\Vert u\Vert=1$. Then
$\Vert (1-u)+\tilde J_x\Vert=1$. For each $n\ge 1$, there exists
$P_n\in {\rm Prim}(A/J_x)$ such that $\Vert (1-u)+\tilde
P_n\Vert>1-1/2^{n+1}$. Hence the set
$$V_n=\{Q\in {\rm Prim}(A): \Vert (1-u)+\tilde
Q\Vert>1-1/2^{n+1}\}$$ is an open neighbourhood of $P_n$ and so,
since $\phi$ is open, the set $\phi(V_n)$ is an open neighbourhood
of $\phi(P_n)=x$ in $X_{\phi}$. Thus $$X_n:=X_{\phi}\setminus
\phi(V_n)=\{ y\in X_{\phi}: \Vert (1-u)+\tilde J_y\Vert\le
1-1/2^{n+1}\}$$ is closed in $X_{\phi}$ and $x\notin X_n$. If $y\in
U_{\phi}$ then $A/J_y$ is unital \cite[Proposition 2.2]{Xalg} and
therefore $\Vert(1-u)+\tilde{J_y}\Vert <1$. Hence
$\bigcup_{n=1}^{\infty}X_n\supseteq U_{\phi}$ and it follows that
$x$ lies in the closure of $\left(\bigcup_{n=2}^{\infty}
X_n\right)\setminus X_1$.

Since $x\notin X_n$ there exists $f_n\in C^b(X_{\phi})$ with $0\le
f_n\le 1/2^n$ such that $f_n(x)=0$ and $f_n(X_n)= \{1/2^n\}$. Set
$f=\sum_{n=1}^{\infty} f_n$. Then $f\in C^b(X_{\phi})$ with $0\le
f\le 1$ and $f(x)=0$. Let $W$ be the cozero set of $f$ and let $V=\{
y\in W: 2\min {\rm sp}(u+{J_y})\le f(y)\}$. Suppose that $y\in
X_{\phi}$ with $y\in X_{n+1}\setminus X_n$ for some $n\ge 1$. Then
$\Vert (1-u)+\tilde J_y\Vert > 1-1/2^{n+1}$ and so $\min {\rm
sp}(u+{J_y})< 1/2^{n+1}$. Since $y\in X_m$ for all $m\geq n+1$,
$f(y)\geq 1/2^n$ and so $f(y)>2\min {\rm sp}(u+{J_y})$. Hence $y\in V$,
and thus $V\supseteq \left(\bigcup_{n=2}^{\infty}
X_n\right)\setminus X_1$. Hence $x\in{\rm cl}(V)\setminus W$ and so
it follows from Theorem~\ref{Thm 2.5}(v) that $\tilde J_x\ne H_x$,
as required.
\end{proof}

%\bigskip

\begin{cor} \label{Cor 3.7} Let $A$ be a continuous,
$\sigma$-unital $C_0(X)$-algebra with base map $\phi$. Then the
following are equivalent:

{\rm (i)} for all $x\in X_{\phi}$, $H_x$ is strictly closed;

{\rm (ii)} $U_{\phi}$ and $W_{\phi}$ are clopen in $X_{\phi}$, and
$W_{\phi}$ is discrete.
\end{cor}

%\bigskip

\begin{proof} (ii)$\Rightarrow$(i). This follows immediately from
Theorem~\ref{Thm 3.6}.

(i)$\Rightarrow$(ii). Theorem~\ref{Thm 3.6} implies that $\partial
U_{\phi}$ is empty and hence $X_{\phi}$ has the required
decomposition into clopen sets $U_{\phi}$ and $W_{\phi}$. Since
$X_{\phi}$ is $\sigma$-compact, $W_{\phi}$ must be $\sigma$-compact
as well. But by Theorem~\ref{Thm 3.6}, $W_{\phi}$ is a P-space, and
a $\sigma$-compact P-space is discrete, see \cite[Lemma
4.4]{multid}.
\end{proof}

%\bigskip

\noindent It is interesting to compare Corollary~\ref{Cor 3.7} with
\cite[Theorem 3.8]{Xalg} which characterizes, for a continuous
$\sigma$-unital $C_0(X)$-algebra $A$, when $M(A)$ is a continuous
$C(\beta X)$-algebra. The conditions in Corollary~\ref{Cor 3.7} are
markedly stronger than those in \cite[Theorem 3.8]{Xalg}, notably
the requirement that $W_{\phi}$ be discrete as against being a
basically disconnected space. On the other hand, if $A$ is separable
then there is a much closer fit with \cite[Corollary 3.9]{Xalg}:
indeed, if $A$ is continuous and separable and $X_{\phi} = X$ then
 $H_x$ is strictly closed for all $x\in X_{\phi}$ if and only if $M(A)$
 is continuous for $\overline\mu$.

\bigskip

 We conclude this section with a couple of elementary abelian
examples, part of whose significance will appear in the next two
sections.

%\bigskip

\begin{example} \label{Ex 3.8} \rm (i)  {\sl An abelian $C_0(X)$-algebra $A$
with $x\in W_{\phi}$ such that $H_x$ is strictly closed in $M(A)$.}
 Let $Y=\{ (x,y)\in {\bf R}^2: y\ge 0\}$ be the upper half-plane,
and let $L=\{ (x,y)\in Y: y=0\}$ be the $x$-axis. Let $Y/L$ be the
quotient space (which is completely regular by Lemma~\ref{Lemma
1.8}). Set $A=C_0(Y)$. Then we may identify ${\rm Prim}(A)$ with $Y$
in the usual way and define $\phi: {\rm Prim}(A)\to \beta (Y/L)$ by
$\phi((x,y))=[(x,y)]\in Y/L$. Thus $X_{\phi}=Y/L$. Then
$J_{[(0,0)]}$ is locally modular (since every point in $ L$ has a
compact neighbourhood in $Y$) and $\phi $ is locally closed at
$[(0,0)]$ (since $Y$ is normal). However, $[(0,0)] \notin U_{\phi} $
because if $f\in C(\beta Y/L)$ and $f\circ\phi\in A=C_0(Y)$ then
$f([(0, 0)]) =0$.

 (ii)  {\sl An abelian $C_0(X)$-algebra $A$ with $x\in X_{\phi}$ such
that $H_x$ is not strictly closed in $M(A)$.} Let $A=C_0([0,1)\cup
[2,3])$ and set $X=[0,1]$. Then we may identify ${\rm Prim}(A)$ with
$[0,1)\cup [2,3]$. Let $\phi: {\rm Prim}(A)\to X$ be given by
$\phi(x)=x$ $(0\le x<1)$ and $\phi(x)=1$ $(2\le x\le 3)$. Then $J_x$
is locally modular for all $x\in X$, but $\phi$ is not locally
closed at $x=1$. Hence $H_1$ is not strictly closed in $M(A)$.
\end{example}

%\bigskip

%\bigskip

\section{$C_0(X)$-algebras where the base map $\phi$ is the complete
regularization map}

%\bigskip

 In this section we investigate $C_0(X)$-algebras $A$ where the base
map $\phi$ is the complete regularization map on ${\rm Prim}(A)$
\cite[Theorem 3.9]{GJ}. Thus $X$ may be taken to be the complete
regularization space (in cases where this is locally compact) or its
Stone-\v{C}ech compactification. Restricting $\phi$ in this way
places a considerable constraint on its behaviour, as we shall see.
This case is of special interest for two reasons: firstly, the
complete regularization map interacts with the topology on ${\rm
Prim}(A)$ in a way that is lacking with more general continuous
maps, and secondly, every continuous map from ${\rm Prim}(A)$ to a
locally compact Hausdorff space factors through the complete
regularization map.

Under the hypotheses that $\phi$ is the complete regularization map
and that $\Orc(A)<\infty$ (a technical assumption which is usually
satisfied), we show that if $A$ is $\sigma$-unital and $J_x$ is
locally modular then $\phi$ is locally closed at $x$ and $H_x$ is
strictly closed (Theorem~\ref{Thm 4.3}). Thus the `locally closed at
$x$' part of the hypothesis in the final sentence of
Proposition~\ref{Prop 2.4} is automatically satisfied in this case
(contrast with Example~\ref{Ex 3.8}(ii)).

%\bigskip

 We begin by explaining the notation $\Orc(A)$ \cite{Som}. Recall
that for a C$^*$-algebra $A$  and for $P, Q\in {\rm Prim}(A)$ we
write $P\sim Q$ if $P$ and $Q$ cannot be separated by disjoint open
sets in ${\rm Prim}(A)$. The relation $\sim$ on ${\rm Prim}(A)$
induces a graph structure on ${\rm Prim}(A)$ whereby $P$ and $Q$ are
adjacent if $P\sim Q$. The distance $d_A(P,Q)$ between $P$ and $Q$
is then defined as the length of the shortest path from $P$ to $Q$
(and is $\infty$ if no such path exists). The diameter of a
$\sim$-component of ${\rm Prim}(A)$ is the supremum of the distances
between primitive ideals in the component (with the convention that
a singleton component, such as when ${\rm Prim}(A)$ is Hausdorff,
has diameter $1$). The {\sl connecting order}, $\Orc(A)$, is the
supremum of the diameters of $\sim$-components of ${\rm Prim}(A)$.
Clearly $\Orc(A)$ is an integer between $1$ and $\infty$, and all
possibilities occur, including $\infty$ \cite[Example 2.8]{Som} (see
also Example~\ref{Ex 5.4}(ii) below). The smaller that $\Orc(A)$ is,
the nearer ${\rm Prim}(A)$ is to being Hausdorff, with the case
$\Orc(A)=1$ corresponding to $\sim$ being an equivalence relation on
${\rm Prim}(A)$. For a subset $Y\subseteq {\rm Prim}(A)$ and  for
$n\ge 0$, let $Y^n=\{ P\in {\rm Prim}(A): \exists\ Q\in Y,\
d_A(P,Q)\le n\}$.

\bigskip

 We also need the following topological lemma characterizing
separation by open sets. We say that a topological space is {\sl
locally compact} if every point has a neighbourhood base of compact
sets.

%\bigskip

\begin{lemma} \label{Lemma 4.1}  Let $X$ be a locally compact topological
space and let $Y$ and $Z$ be subsets of $X$ which are Lindelof in
the relative topology. Then the following are equivalent:

{\rm (i)} The closure of $Y^1$ does not meet $Z$ and the closure of
$Z^1$ does not meet $Y$.

{\rm (ii)} There exist disjoint open subsets $U$ and $V$ of $X$ with
$Y\subseteq U$ and $Z\subseteq V$.
\end{lemma}

%\bigskip

\begin{proof} Suppose first that (ii) holds. Then $X\setminus V$ is
disjoint from $Z$ and is a closed set containing the neighbourhood
$U$ of $Y$ and hence containing $Y^1$. Similarly $X\setminus U$ is
disjoint from $Y$ and is a closed set containing $Z^1$. Hence (i)
holds.

Conversely, suppose that (i) holds. Let $x\in Y$. Since the closure of $Z^1$
does not meet $Y$, $x$ has an open neighbourhood disjoint from $Z^1$.
Hence, by the local compactness of
$X$, $x$ has a compact neighbourhood $U_x\subseteq X\setminus Z^1$. Then $U_x^1$ is
closed
because $U_x$ is compact, and $U_x^1$ does not meet $Z$ because $U_x$ does not
meet $Z^1$.
Similarly for each $x\in Z$ there exists a
neighbourhood $V_x$ of $x$ such that the closure of $V_x$ does not meet $Y$.

Since $Y$ and $Z$ are Lindelof, we may obtain a countable
collection, say $U_1, U_2, \ldots$, of the sets ${\rm Int}\ U_x$
such that $\bigcup_{i=1}^{\infty} U_i$ covers $Y$, and likewise a
countable collection $V_1, V_2, \ldots$ of the sets ${\rm Int}\ V_x$
such that $\bigcup_{i=1}^{\infty} V_i$ covers $Z$. For each $i\ge
1$, set $U'_i=U_i\setminus \left(\bigcup_{j=1}^i \overline
V_j\right)$ and $V'_i=V_i\setminus \left(\bigcup_{j=1}^i \overline
U_j\right)$ (where $\overline V_i$ denotes the closure of $V_i$,
etc). Set $U=\bigcup_{i=1}^{\infty} U'_i$ and
$V=\bigcup_{i=1}^{\infty} V'_i$. Then it is easily checked that
$Y\subseteq U$ and $Z\subseteq V$. If $x\in U\cap V$ then there
exist $U'_i$ and $V'_j$ such that $x\in U'_i\cap V'_j$. Without loss
of generality we may suppose that $i\ge j$. But then $V'_j\subseteq
V_j$ which is disjoint from $U'_i$, a contradiction. Hence $U$ and
$V$ are disjoint, and thus (ii) holds.
\end{proof}

%\bigskip

 Now let $A$ be a $\sigma$-unital $C_0(X)$-algebra. Let $E$ be a
non-empty closed subset of $X_{\phi}$, set $Y=X_{\phi}/E$ and let
$q:X_{\phi}\to Y$ be the quotient map. Set $\psi=q\circ\phi$ and
$\{e\}=q(E)$. We saw in Proposition~\ref{Prop 1.9} that the question
of spectral synthesis for the set $E$ can be reduced to that of the
point $e$, and for this reason we have previously confined ourselves
to considering spectral synthesis at points. If we restrict $\phi$
to be the complete regularization map, however, then we can no
longer make this reduction (because the reduction changes the base
map), and we will therefore have to work with closed sets in this
section.

With this in mind, and with the notation above,
we say that $\phi$ is {\sl locally closed at $E$} if $\psi$ is locally
closed at $e$, and that $J_E$ is {\sl locally modular} if $J_e$ is locally modular.
Elementary topological arguments show that $\phi$ is locally closed at $E$ if and only if
whenever $Y$ is a closed subset of $\Prim(A)$ such that  $\phi(Y) \cap E = \emptyset$
(i.e. $Y \cap \phi^{-1}(E) = \emptyset$ )  then  $\overline {\phi(Y)} \cap E = \emptyset$.
Note that $H(e) = (q\circ \phi)^{-1}(e) = \phi^{-1}(E)$, the hull of $J_E$ in Prim(A), and
recall that $U(e)$ is the interior of $H(e)$ in $\Prim(A)$.
For $Q, R\in \Prim(M(A))\setminus \tilde U(e)$, recall that we write $Q\sim_e R$ if
there
is a net $(P_{\alpha})$ in $\Prim(A)\setminus U(e)$ such that $\tilde P_{\alpha}\to Q, R$.

%\bigskip

\begin{prop} \label{Prop 4.2}  Let $A$ be a $\sigma$-unital
$C_0(X)$-algebra with base map $\phi$ and let $E$ be a non-empty
closed subset of $X_{\phi}$. Consider the following conditions

{\rm (i)} $H_E$ is strictly closed.

{\rm (ii)} $J_E$ is locally modular and $\phi$ is locally closed at
$E$.

\noindent Then {\rm (ii)}$\Rightarrow${\rm (i)}, and {\rm (i)} and
{\rm (ii)} are equivalent if $A$ is separable.
\end{prop}

%\bigskip

\begin{proof} With the notation above, we have $J_e=J_E$ and $H_e=H_E$
by Proposition~\ref{Prop 1.9}. Hence (ii)$\Rightarrow$(i) follows
immediately from Proposition~\ref{Prop 2.4}.

Conversely, suppose that $A$ is separable and that (i) holds. Then
$H_e$ is strictly closed by Proposition~\ref{Prop 1.9} and it
follows from Corollary~\ref{Cor 2.10} that $J_e=J_E$ is locally
modular and $q\circ\phi$ is locally closed at $e$. Hence $\phi$ is
locally closed at $E$ by definition.
\end{proof}

%\bigskip

 We are now ready for the main theorem of the section.

%\bigskip

\begin{thm} \label{Thm 4.3}  Let $A$ be a $\sigma$-unital
$C_0(X)$-algebra and suppose that $\phi$ is the complete
regularization map for ${\rm Prim}(A)$ and that $\Orc(A)<\infty$.
Let $E$ be a non-empty closed subset of $X_{\phi}$ and suppose that
$J_E$ is locally modular. Then $\phi$ is locally closed at $E$ and
$H_E$ is strictly closed.
\end{thm}

%\bigskip

\begin{proof} As before, let $q:X_{\phi}\to X_{\phi}/E$ be the
quotient map and $\{e\}=q(E)$. By Proposition~\ref{Prop 4.2}, it
suffices to show that $\phi$ is locally closed at $E$. Let $V$ be a
proper open subset of ${\rm Prim}(A)$ with $V\supseteq H(e)$. It
suffices to produce a continuous function $g$ on ${\rm Prim}(A)$
with $g({\rm Prim}(A)\setminus V)=\{ 1\}$ and $g(H(e))=\{0\}$ (for
then $g$ induces a continuous function on $\phi(\Prim(A)$)
separating $\phi(\Prim(A)\setminus V)$  and $E$, so that
$\overline{\phi(\Prim(A)\setminus V)} \cap E = \emptyset$).

Set $Y=\{ Q\in {\rm Prim}(A)\setminus U(e): \exists R\in {\rm
Prim}(M(A)/A)\hbox{ with } \tilde Q\sim_e R\}$. Then $Y\cap H(e)$ is
empty by Lemma~\ref{Lemma 2.3}. We claim that $Y$ is a closed subset
of ${\rm Prim}(A)$. To see this, let $(Q_{\alpha})$ be a net in $Y$
with $Q_{\alpha}\to Q\in {\rm Prim}(A)$. Then $Q\notin U(e)$ since
${\rm Prim}(A)\setminus U(e)$ is closed. By definition of $Y$, there
is a net $(R_{\alpha})$ in ${\rm Prim}(M(A)/A)$ with $\tilde
Q_{\alpha}\sim_e R_{\alpha}$ for each ${\alpha}$. By the compactness
of ${\rm Prim}(M(A)/A)$, there is a subnet $(R_{\beta})$ of
$(R_{\alpha})$ with $R_{\beta}\to R\in {\rm Prim}(M(A)/A)$. Then
$\tilde Q\sim_e R$ and so $Q\in Y$, as required.

Next we note that $Y^1$ is closed in ${\rm Prim}(A)$. To see this, let $(Q_{\alpha})$ be a net in $Y^1$
with $Q_{\alpha}\to Q\in {\rm Prim}(A)$. Then $Q\notin U(e)$,
for otherwise eventually $Q_{\alpha}\in U(e)\subseteq
H(e)$, which is impossible since $H(e)$ is a $\sim$-saturated set disjoint from $Y$.
Let $(P_{\alpha})$ be a net in $Y$ such that $P_{\alpha}\sim Q_{\alpha}$
for each $\alpha$. Since $P_{\alpha}, Q_{\alpha}\notin H(e)$,
$\tilde P_\alpha\sim_e \tilde Q_\alpha$.
By the compactness of ${\rm Prim}(M(A))$ there exists $R\in {\rm Prim}(M(A))$ and
a subnet $(P_{\beta})$
of $(P_{\alpha})$ such that $\tilde P_\beta\to R$. Hence $R\sim_e \tilde Q$.
If $R=\tilde S$ for some $S\in Y$
then $Q\in Y^1$ as required. Otherwise $R\in {\rm Prim}(M(A)/A)$ and hence $Q\in Y\subseteq Y^1$. Thus $Y^1$ is closed.

Since $A$ is $\sigma$-unital, $\Prim(A)$ is $\sigma$-compact and
hence every closed subset of $\Prim(A)$ is a Lindelof space. Thus,
since $H(e)$ is closed and $\sim$-saturated, we may apply
Lemma~\ref{Lemma 4.1} to $H(e)$ and $Y$ to obtain disjoint open
subsets $V'$ and $V''$ of ${\rm Prim}(A)$ with $H(e)\subseteq V'$
and $Y\subseteq V''$. By intersecting $V'$ with the open set $V$ of
the first paragraph, we may assume that $V'\subseteq V$. Set $Z={\rm
Prim}(A)\setminus V'$. Then $Z\supseteq V''\supseteq Y$ and the
boundary of $Z$ in ${\rm Prim}(A)$ does not meet $V''$. Let
$k=\Orc(A)$. Then the same argument that showed that $Y^1$ is closed
also shows, inductively, that $Z^1, \ldots, Z^k$ are closed. Note
that if $Q$ belongs to the boundary of $Z^k$ in ${\rm Prim}(A)$ then
$Q\notin V''$ and so $Q\notin Y$. Thus there does not exist $R\in
{\rm Prim}(M(A)/A)$ with $\tilde Q\sim_e R$.

Note that the disjoint sets $Z^k$  and $H(e)$ are $\sim$-saturated and hence so is the set
$F:= {\rm Prim}(A)\setminus (Z^k\cup H(e))$.
We now define an equivalence relation $\diamond$ on ${\rm Prim(A)}$ as follows. The $\diamond$-equivalence
classes are: $Z^k$, $H(e)$, and the $\sim$-components of $F$.
Set $W={\rm Prim}(A)/\diamond$, equipped with the quotient topology,
and let $p:{\rm Prim}(A)\to W$ be the quotient map. We
show that $p$ is a closed map. Let $T$ be a closed subset of ${\rm Prim}(A)$
and set $T'= p^{-1}(p(T))$. Let $(Q_{\alpha})$ be a net in $T'$ with a limit $Q\in {\rm Prim}(A)$.
We must show that $Q\in T'$. For each $Q_{\alpha}$ there exists $R_{\alpha}\in T$ such
that $Q_{\alpha}\diamond R_{\alpha}$, and by the compactness of ${\rm Prim}(M(A))$ and
by passing to subnets of $(Q_{\alpha})$ and $(R_{\alpha})$, if necessary, we may
assume that there exists $R\in {\rm Prim}(M(A))$ with $\tilde R_{\alpha}\to R$.

We consider various cases. If $(Q_{\alpha})$ is frequently in $Z^k$
then $Q\in Z^k$, since $Z^k$ is closed. Hence $Q\diamond Q_{\alpha}$
for $Q_{\alpha}\in Z^k$ and so $Q\in T'$ since $T'$ is
$\diamond$-saturated. A similar argument shows that $Q\in T'$ if
$(Q_{\alpha})$ is frequently in $H(e)$. Hence we may restrict
attention to the case when $Q_{\alpha}\in F$ for all $\alpha$. This
implies that $d_A(Q_{\alpha}, R_{\alpha})\le k$ and hence that for
each $\alpha$ there is a walk $Q_{\alpha}\sim
Q^1_{\alpha}\sim\ldots\sim Q^k_{\alpha}=R_{\alpha}$ (possibly with
repetitions) of length $k$ between $Q_{\alpha}$ and $R_{\alpha}$.
Hence $\tilde Q_{\alpha}\sim_e \tilde Q^1_{\alpha}\sim_e\ldots\sim_e
\tilde Q^k_{\alpha}=\tilde R_{\alpha}$. Using the compactness of
${\rm Prim}(M(A))\setminus\tilde U(e)$, and passing to successive
subnets, we obtain a walk $\tilde Q\sim_e Q^1\sim_e\ldots\sim_e
Q^k=R$ of length $k$ in ${\rm Prim}(M(A))\setminus \tilde U(e)$ such
that $\tilde Q, Q^1,\ldots, Q^k$ all lie in $(\tilde F)^{-}$, the
closure of $\tilde F$ in ${\rm Prim}(M(A))$.

Suppose that $P\in F^{-}$ (the closure of $F$ in ${\rm Prim}(A)$)
and that $\tilde P\sim_e P'$ for some $P'\in (\tilde F)^{-}$. Then
$P\notin U(e)$ and $P\notin Y$ and hence, by the definition of $Y$,
$P'=\tilde S$ for some $S\in {\rm Prim}(A)$. Furthermore, $S\in
F^{-}$ because $\tilde S\in (\tilde F)^{-}$. Since $Q\in F^{-}$, it
follows by induction that $Q^i=\tilde S_i$ $(1\le i\le k)$ for some
$S_i\in {\rm Prim}(A)$, and hence that $Q\diamond S_k$. But $S_k\in
T$, since $T$ is closed in ${\rm Prim}(A)$, and hence $Q\in T'$ as
required. Thus we have shown that $p$ is a closed map.

Now let $C$ and $D$ be non-empty, disjoint closed subsets of $W$.
Then $C':=p^{-1}(C)$ and $D':=p^{-1}(D)$ are disjoint closed
$\sim$-saturated subsets of ${\rm Prim}(A)$. Thus $C'$ and $D'$ are
Lindelof and so Lemma~\ref{Lemma 4.1} implies the existence of
disjoint open sets $E'$ and $F'$ containing $C'$ and $D'$
respectively. We now use a standard characterization (see e.g.
\cite[$\S$13.XIV Theorem 3]{Kur}): a quotient map $r:M\to N$ is
closed if and only if whenever $d\in N$ and $O$ is an open set
containing $r^{-1}(d)$ then there exists a saturated open set $H$
such that $r^{-1}(d)\subseteq H\subseteq O$ (where $H$ is saturated
if $H=r^{-1}(r(H))$). Applying this characterization in the present
case to each of the points of $C$ and $D$ relative to $E'$ and $F'$
we obtain $\diamond$-saturated open sets $E''$ and $F''$ such that
$C'\subseteq E''\subseteq E'$ and $D'\subseteq F''\subseteq F'$.
Hence $p(E'')$ and $p(F'')$ are disjoint open sets of $W$ containing
$C$ and $D$ respectively. Thus $W$ is normal.

It follows that there is a positive continuous function $f$ on $W$
with $\Vert f\Vert=1$ such that $f(p(Z^k))=\{1\}$ and
$f(p(H(e)))=\{0\}$. Then $g=f\circ p$ is a continuous function on
${\rm Prim}(A)$ with $g(Z^k)= \{1\}$ and $g(H(e))=\{0\}$. Since
${\rm Prim}(A)\setminus V\subseteq Z^k$, $g$ has the property
required at the start of the proof.
\end{proof}

%\bigskip

 Now let $A$ be a $C_0(X)$-algebra with base map $\phi$. Then
$\phi$ factors as $\phi=\psi\circ\phi_A$ where $\phi_A$ is the
complete regularization map and $\psi$ is continuous. The advantage
of the following result over Proposition~\ref{Prop 4.2} is that
$\psi$ is a map between completely regular spaces and should
therefore be simpler to analyze.

Let $X_{\phi_A}$ denote the image of $\Prim(A)$ under the complete regularization
map $\phi_A$. Then $\psi$ is a map from $X_{\phi_A}\to X_{\phi}$.
By analogy with our earlier definition, we say that $\psi$ is locally closed at a
non-empty subset $E\subseteq X_{\phi}$ if
whenever $W$ is a closed subset of $X_{\phi_A}$ with $\psi(W) \cap E = \emptyset$
(i.e. $W \cap \psi^{-1}(E) = \emptyset$) then $\overline{\psi(W)} \cap E = \emptyset$.

%\bigskip

\begin{cor} \label{Cor 4.4}  Let $A$ be a $\sigma$-unital
$C_0(X)$-algebra with base map $\phi$ and suppose that
$\Orc(A)<\infty$. Write $\phi=\psi\circ\phi_A$ where $\phi_A$ is the
complete regularization map for ${\rm Prim}(A)$. Let $E$ be a
non-empty closed subset of $X_{\phi}$. Consider the following
conditions:

{\rm (i)} $H_E$ is strictly closed;

{\rm (ii)} $J_E$ is locally modular and $\psi:X_{\phi_A}\to
X_{\phi}$ is locally closed at $E$.

\noindent Then {\rm (ii)}$\Rightarrow${\rm (i)}, and {\rm (i)} and
{\rm (ii)} are equivalent if $A$ is separable.
\end{cor}

%\bigskip

\begin{proof} Set $H(E)=\phi^{-1}(E)$. First, suppose that (ii) holds
and let $Y$ be a closed  subset of ${\rm Prim}(A)$ with $Y\cap H(E)$
empty. Let $W$ be the closure of $\phi_A(Y)$ in $X_{\phi_A}$. Then
$W$ does not meet $\phi_A(H(E))=\psi^{-1}(E)$ by Theorem~\ref{Thm
4.3}. Hence $\psi(W) \cap E = \emptyset$ and so the closure of
$\psi(W)$ in $X_{\phi}$ does not meet $E$, since $\psi$ is locally
closed at $E$. But $\phi(Y)\subseteq \psi(W)$ and hence $\overline
{\phi(Y)} \cap E = \emptyset$. Thus $\phi$ is locally closed at E
and so $J_E$ is strictly closed by Proposition~\ref{Prop 4.2}.

Conversely, suppose that (i) holds and that $A$ is separable. Then
$J_E$ is locally modular and $\phi$ is locally closed at $E$ by
Proposition~\ref{Prop 4.2}. Let $W$ be a closed set in $X_{\phi_A}$
such that $W\cap \psi^{-1}(E)$ is empty. Set $Y=\phi_A^{-1}(W)$.
Then $Y$ is closed in ${\rm Prim}(A)$ and $Y\cap H(E)$ is empty.
Hence the closure of $\phi(Y)$ does not meet $E$. But
$\phi(Y)=\psi(W)$ and hence $\overline{\psi(W)} \cap E = \emptyset$.
Thus $\psi$ is locally closed at $E$.
\end{proof}

%\bigskip

\noindent In particular, if $A$ in Corollary~\ref{Cor 4.4} is
separable and $\psi$ is a closed map (for example, the identity map
when $\phi=\phi_A$) then $H_E$ is strictly closed if and only if
$J_E$ is locally modular.

%\bigskip

%\bigskip

\section{Locally modular ideals}

%\bigskip

 In this final section we look at locally modular ideals in the
case when $\phi$ is the complete regularization map and
$\Orc(A)<\infty$. We saw immediately after the definition of local
modularity that there are two `easy' ways for $J_x$ to be locally
modular: if $x\in U_{\phi}$ or if $x$ is an isolated point in
$X_{\phi}$. Example~\ref{Ex 3.8} gave two examples where $J_x$ is
locally modular with $x\in\partial U_{\phi}$, and in the first of
these $H(x)$ has empty interior in $\Prim(A)$. We will show that
such behaviour cannot occur when $\phi$ is the complete
regularization map and $\Orc(A)<\infty$. In this case, if $J_x$ is
locally modular then either $x\in U_{\phi}$ or $H(x)$ has non-empty
interior (Corollary~\ref{Cor 5.3}).

Recall that for a C$^*$-algebra $A$, we say that $P, Q\in\Prim(A)$
belong to the same {\sl Glimm class} if $f(P)=f(Q)$ for all
continuous, bounded, real-valued functions $f$ on $\Prim(A)$
(equivalently, $\phi_A(P) = \phi_A(Q)$, where $\phi_A$ is the
complete regularization map on $\Prim(A)$). The algebra $A+Z(M(A))$
in the next result was introduced by Dixmier \cite{DixIdeal}.

%\bigskip

\begin{lemma} \label{Lemma 5.1}  Let $A$ be a C$^*$-algebra with
$\Orc(A)<\infty$ and let $C=A+Z(M(A))$. Then $\Orc(C)\le 2\
\Orc(A)+2$.
\end{lemma}

%\bigskip

\begin{proof} First note that $A$ is an essential ideal in $C$, so
that ${\rm Prim}(A)$ is (homeomorphic to) a dense open subset of
${\rm Prim}(C)$. Suppose that $Q_1$, $Q_2$ are distinct elements of
${\rm Prim}(C)\setminus {\rm Prim}(A)$. Then $Q_i=M_i+A$ where $M_i$
is a maximal ideal of $Z(M(A))$ containing $A\cap Z$ $(i=1, 2)$. It
follows that $Z(M(A))$ separates $Q_1$ and $Q_2$. Thus each Glimm
class in ${\rm Prim}(C)$ contains at most one element of ${\rm
Prim}(C)\setminus {\rm Prim}(A)$. Hence if $P_1\sim
P_2\sim\ldots\sim P_n$ is a path in ${\rm Prim}(C)$ then at most one
element from ${\rm Prim}(C)\setminus {\rm Prim}(A)$ can occur among
the $P_i$'s. It follows that $d_C(P_1, P_n)\le 2\ \Orc(A)+2$ as
required.
\end{proof}

%\bigskip

 The next theorem is a general result giving a dichotomy for $\sim$-components in $\Prim(A)$
 for any
 $C^*$-algebra $A$ for which $\Orc(A)<\infty$.

%\bigskip

\begin{thm} \label{Thm 5.2}  Let $A$ be a $C_0(X)$-algebra and
suppose that $\phi$ is the complete regularization map for ${\rm
Prim}(A)$ and that $\Orc(A)<\infty$. Let $T$ be a $\sim$-component
of ${\rm Prim}(A)$, so that $T\subseteq H(x)$ for some $x\in
X_{\phi}$. Then either

{\rm (i)} $J_x\not\supseteq Z(A)$ and $T=H(x)$; or

{\rm (ii)} $J_x\supseteq Z(A)$ and there exist $P\in T$ and $R\in
{\rm Prim}(M(A))$ with $R\supseteq A$ such that $\tilde P\sim R$.
\end{thm}

%\bigskip

\begin{proof} Suppose that $P\in T$ and $R\in {\rm Prim}(M(A))$
with $R\supseteq A$ and $\tilde P\sim R$. Then
$\overline\phi(R)=\overline\phi(\tilde P)=\phi(P)=x$ and it follows
from Lemma~\ref{Lemma 1.6} that $J_x\supseteq Z(A)$. We must show,
therefore, that if there do not exist $P\in T$ and $R\in {\rm
Prim}(M(A))$ with $R\supseteq A$ such that $\tilde P\sim R$ then
alternative (i) holds.

Suppose, then, that for all $P\in T$ and $R\in {\rm Prim}(M(A))$
with $R\supseteq A$, $\tilde P\not\sim R$. Set $k=\Orc(A)$, and let
$Q\in T$. Note that $\tilde T=\{ R\in {\rm Prim}(M(A)):
d_{M(A)}(\tilde Q, R)\le k\}$, by the supposition that for all $P\in
T$ and $R\in {\rm Prim}(M(A))$ with $R\supseteq A$, $\tilde
P\not\sim R$. Hence $\tilde T$ is a closed (and thus compact) subset
of ${\rm Prim}(M(A))$ by \cite[Corollary 2.3]{Som} applied $k$
times. Set $L=\ker\tilde T$, a closed ideal of $M(A)$. If $A+L$ were
a proper ideal of $M(A)$ there would exist $R\in \Prim(M(A))$ such
that $R \supseteq A+L$. Hence $R\supseteq L$ and so $R\in \tilde T$
since $\tilde T$ is a closed subset of ${\rm Prim}(M(A))$, but also
$R\supseteq A$. This is a contradiction, and hence $A+L=M(A)$.

Now set $C=A+Z(M(A))$, and for each $P\in {\rm Prim}(A)$ let $P^{\prime}$
be the unique primitive ideal in  $C$ such that $P^{\prime}\cap A=P$.

Let $T^{\prime}:=\{P^{\prime} : P\in T\}$. We claim that $K\not\sim
P^{\prime}$ whenever $P\in T$ and $K\in {\rm Prim}(C)$ with
$K\supseteq A$. Supposing otherwise, there exist $P\in T$, $K\in
{\rm Prim}(C)$ with $K\supseteq A$, and a net $(P_{\alpha})$ in
${\rm Prim}(A)$ such that $P^{\prime}_{\alpha}\to P^{\prime}$ and
$P^{\prime}_{\alpha}\to K$. Since $K \supseteq A$ and $C\subseteq
M(A) = A+L$, $K \not\supseteq L\cap C$. Hence there exists $c\in
L\cap C$ such that $\Vert c+K\Vert=1$. By lower semi-continuity,
$\Vert c+P^{\prime}_{\alpha} \Vert\ge 1/2$ eventually. Hence $\Vert
c+\tilde P_{\alpha}\Vert\ge 1/2$ eventually (because both $\Vert
c+P^{\prime}_{\alpha}\Vert$ and $\Vert c+\tilde P_{\alpha}\Vert$ are
equal to $\Vert\overline\pi_{\alpha}(c)\Vert$ where $\pi_{\alpha}$
is an irreducible representation of $A$ with kernel $P_{\alpha}$ and
$\overline\pi_{\alpha}$ is its unique ultra-weakly continuous
extension to $A^{**}$). By the compactness of the set $\{S\in {\rm
Prim}(M(A)): \Vert c+S\Vert\ge 1/2\}$, we may assume, without loss
of generality, that $\tilde P_{\alpha}\to R$ in ${\rm Prim}(M(A))$,
where $\Vert c+R\Vert\ge 1/2$.

Since $P^{\prime}_{\alpha}\to P^{\prime}$ we have $P_{\alpha}\to P$
and $\tilde P_{\alpha}\to \tilde P$. Thus $\tilde P\sim R$. By the
supposition of the second paragraph, it must be that $R=\tilde S$
for some $S\in {\rm Prim}(A)$. Hence $S\in T$.
Thus we have  $c\in L \subseteq \tilde{S} = R$, contradicting the fact that
$\Vert c+R\Vert\ge 1/2$. It follows, then, that
$K\not\sim P^{\prime}$ whenever $P\in T$ and $K\in {\rm Prim}(C)$ with
$K\supseteq A$, and hence that $T^{\prime}$ is a $\sim$-component in
${\rm Prim}(C)$.

By Lemma~\ref{Lemma 5.1}, $\Orc(C)\le 2\ \Orc(A)+2<\infty$, and
since ${\rm Prim}(C)$ is compact it follows that $T^{\prime}$ is a
Glimm class in ${\rm Prim}(C)$ \cite[Corollary 2.7]{Som}. It follows
at once that $T$ is a Glimm class in ${\rm Prim}(A)$. Thus $T=H(x)$.
Now, let $P\in T$, and let $\phi_C: {\rm Prim}(C)\to Y$ be the
complete regularization map for ${\rm Prim}(C)$, where $Y$ is the
space of Glimm ideals of $C$ with the induced completely regular
topology. Then the set $W=\phi_C({\rm Prim}(C)\setminus {\rm
Prim}(A))$ is compact and hence closed in $Y$, and does not contain
the point $\phi_C(P')$. Thus there exists a continuous function $f$
on $Y$ taking the value $1$ on $\phi_C (P^{\prime})$ and $0$ on $W$.
Since $f$ vanishes on $W$, it follows that the central element $z$
of $C$ induced by $f$ actually belongs to $A$. But $z\notin
P^{\prime}$, and $P^{\prime}\supseteq P\supseteq J_x$, and thus
$z\notin J_x$. Hence $J_x\not\supseteq Z(A)$.
\end{proof}

%\bigskip

\begin{cor} \label{Cor 5.3}  Let $A$ be a $C_0(X)$-algebra and
suppose that $\phi$ is the complete regularization map for ${\rm
Prim}(A)$ and that $\Orc(A)<\infty$. Let $x\in X_{\phi}$ with $J_x$ locally modular. Then either $x\in U_{\phi}$ or $H(x)$ has
non-empty interior.
\end{cor}

%\bigskip

\begin{proof} Suppose that $x\notin U_{\phi}$ and that $H(x)$ has
empty interior. Let $T$ be a $\sim$-component of $H(x)$. By
Theorem~\ref{Thm 5.2} there exists $P\in T$ and $R\in {\rm
Prim}(M(A)/A)$ such that $\tilde P\sim R$. Since $H(x)$ has empty
interior, $P\in
\partial H(x)$ and $\tilde P\sim_x R$. By Lemma~\ref{Lemma 2.3}, $J_x$ is not
locally modular.
\end{proof}

%\bigskip

 We conclude with two examples of $x\in\partial U_{\phi}$ with
$J_x$ locally modular. The first has $\Orc(A)<\infty$ and $H(x)$
with non-empty interior. The second has $\Orc(A)=\infty$ and $H(x)$
with empty interior, showing that the condition $\Orc(A)<\infty$ in
Corollary~\ref{Cor 5.3} is not redundant.

%\bigskip

\begin{example} \label{Ex 5.4} (i) {\sl A $C_0(X)$-algebra with $z\in
\partial U_{\phi}$ such that $J_z$ is locally modular, $\phi$ is locally
closed at $z$, and $H(z)$ has non-empty interior.}

As in Example~\ref{Ex 3.8}(i), let $Y=\{ (x,y)\in {\bf R}^2: y\ge
0\}$ be the upper half-plane, and let $L=\{ (x,y)\in Y: y=0\}$ be
the $x$-axis. Set $B=C_0(Y)$ and $C=C_0(L)$, and let $\pi: B\to C$
be the surjective $*$-homomorphism given by $\pi(b)=b|_L$ $(b\in
B)$. Let $H$ be a separable, infinite-dimensional Hilbert space,
$B(H)$ the algebra of bounded operators on $H$, and $K(H)$ the
algebra of compact operators on $H$. Let $\rho: C\to B(H)$ be a
$^*$-monomorphism such that $\rho(C)\cap K(H)=\{0\}$. Set
$D=\rho(C)+K(H)$, a C$^*$-subalgebra of $B(H)$. Note that each
element $d\in D$ can be uniquely expressed in the form $d=g+T$ where
$g\in \rho(C)$ and $T\in K(H)$.

Set $A=\{ (b,d)\in B\oplus D: \rho(\pi(b))=g\}$.
Then $A$ is separable. For $(x,y)\in Y$, let $\theta_{x, y}$ be the character on
$A$ given by $\theta_{x,y}(b,d)=b((x,y))$. Set $G=\{ (b,d)\in A:
\pi(b)=0,\ T=0\}$. Since any irreducible representation of $A$ extends
to an irreducible representation of
$B\oplus D$ (on a possibly larger Hilbert space),
${\rm Prim}(A)=\{\ker\theta_{x,y}: (x,y)\in
Y\}\cup\{ G\}$. Note that $G \subseteq \ker\theta_{x,0}$ for all $x\in R$.
It follows that a subset $W$ of $\Prim(A)$ is closed if and only if
(i) $\{(x,y)\in Y: \ker\theta_{x,y}\in W \} \cap Y$ is closed in $Y$ (with the usual topology),
and (ii) if $G\in W$ then $\ker\theta_{x,0} \in W$ for all $x\in {\bf R}$.
In particular $\{G\}$ is an open subset of $\Prim(A)$.

Set $X_{\phi}=Y/L$ and let $q:Y\to X_{\phi}$ be the quotient map.
Set $X=\beta X_{\phi}$. Define $\phi: {\rm Prim}(A)\to
X_{\phi}\subseteq X$ by $\phi(\theta_{x,y})= q(x,y)$ $((x,y)\in Y)$
and $\phi(G)=q(0,0)$. Then $\phi$ is the complete regularization map
for ${\rm Prim}(A)$ and $\phi(G)$ is non-isolated in $X_{\phi}$. For
$(x,y)\in Y\setminus L$, $J_{q(x,y)}=\ker\theta_{x,y}$ while
$J_{q(0,0)} = G$. Each point of $Y$ has a compact neighbourhood in
$Y$ and hence $J_x$ is locally modular for each $x\in X_{\phi}$,
although $A/G$ is non-unital. Since $Y$ is normal, the map $\phi$ is
easily seen to be relatively closed and hence $H_x$ is strictly
closed for each $x\in X_{\phi}$ by Proposition~\ref{Prop 2.4}.
Taking $z=q(0,0) = q(G)$, however, we have that $$H(q(0,0)) =
\{\ker\theta_{x,0}:x\in {\bf R}\} \cup \{G\}$$ and this has
non-empty interior $\{G\}$.

\medskip

(ii) {\sl A $C_0(X)$-algebra with $x\in
\partial U_{\phi}$ such that $J_x$ is locally modular, $\phi$ is
locally closed at $x$, and $H(x)$ has empty interior.}

Let $A$ be the C$^*$-algebra defined as follows (see \cite[Example
2.8]{Som}). Let $B$ be the C$^*$-algebra consisting of all
continuous functions from the interval $[0,1]$ into the $2\times 2$
complex matrices. Let $B(1)$ be the C$^*$-subalgebra of $B$
consisting of those functions $f\in B$ satisfying $f(2^{-n})={\rm
diag}(\lambda_{2n-1}(f), \lambda_{2n}(f))$, $(n\ge 1)$, and
$f(0)={\rm diag}(\lambda(f),\lambda(f))$, for some complex numbers
$\lambda(f)$, $\lambda_n(f)$ $(n\ge 1)$. Let $A=\{f\in B(1):
\lambda_{2n}(f)=\lambda_{2n+1}(f)\quad (1\le n<\infty) \hbox{ and
}\lambda(f)=0\}$. Then $A$ is separable.

For $y\in (0,1]\setminus \{2^{-n}: n\ge 1\}$, set $P_y=\{ f\in A:
f(y)=(0)\}$. Then
 $${\rm Prim}(A)=\{ P_y: y\in (0,1]\setminus
\{2^{-n}: n\ge 1\}\}\cup \{\ker(\lambda_i): i=1, 3, 5,\ldots \}.$$
Set $X_{\phi}=(0,1]/*$ where for $r,s\in (0,1]$, $r*s$ if $r=s$ or
if $r,s\in\{2^{-n}: n\ge 1\}$ and let $q:(0,1]\to X_{\phi}$ denote
the quotient map. Set $X=\beta X_{\phi}$ and $\infty=q(1/2)$. Define
$\phi: {\rm Prim}(A)\to X$ by $\phi(P_y)=y$ $(y\in (0,1]\setminus
\{2^{-n}: n\ge 1\})$ and $\phi(\lambda_i)=\infty$ $(i=1,3, 5,
\ldots)$. Then $\phi$ is the complete regularization map for ${\rm
Prim}(A)$.

If  $ x \in X_{\phi} \setminus \{\infty\} $ then $J_x \not\supseteq
Z(A) = Z'(A)$ and hence $x\in U_{\phi}$ and $J_x$ is locally
modular. It is easy to see directly that $J_{\infty}$ is also
locally modular. But $A/J_{\infty}$ is non-unital, since ${\rm
Prim}(A/J_{\infty})=\{ \ker(\lambda_i): i=1, 3, 5, \ldots\}$ is
non-compact, and hence $J_{\infty}\supseteq Z(A)$. Thus $U_{\phi}=
X_{\phi}\setminus \{\infty\}$ and $W_{\phi}=\{\infty\}$. We show
that $\phi$ is a relatively closed map. Let $Y$ be a closed subset
of ${\rm Prim}(A)$ and set $Y'=\phi^{-1}( \phi(Y))$. Then $Y'=Y$ if
$Y\cap \{ \ker(\lambda_i): i=1, 3, 5, \ldots\}$ is empty, and
$Y'=Y\cup \{ \ker(\lambda_i): i=1, 3, 5, \ldots\}$ otherwise. In
either case $Y'$ is closed, and hence $\phi$ is relatively closed.
It follows, therefore, from Proposition~\ref{Prop 2.4} that $H_x$ is
strictly closed for every $x\in X_{\phi}$.
\end{example}

%\centerline {\bf References}

%\input TempReferences

\end{document}